\documentclass[10pt,a4paper]{article}
\usepackage[a4paper]{geometry}
\usepackage{amssymb,latexsym,amsmath,amsfonts,amsthm}
\usepackage{graphicx,comment}
\usepackage{epsfig}
\usepackage{tikz}
\usepackage{fontenc}

\newcommand{\Tr}{\mathrm{Tr}\,}

\newcommand{\card}{\mathrm{card}\,}

\newcommand{\ov}{\overline}

\newtheorem{theorem}{Theorem}[section]
\newtheorem{lemma}[theorem]{Lemma}
\newtheorem{proposition}[theorem]{Proposition}

\newtheorem{corollary}[theorem]{Corollary}

\theoremstyle{definition}

\theoremstyle{remark}

\numberwithin{equation}{section}

\title{On random polynomials generated by a symmetric three-term recurrence relation}

\date{\today}

\author{Abey L\'{o}pez-Garc\'{i}a\footnotemark[1] \qquad Vasiliy A. Prokhorov\footnotemark[2]}

\begin{document}

\maketitle

\par{\centering Dedicated to Guillermo L\'{o}pez Lagomasino, in celebration of his 70th birthday \par}

\renewcommand{\thefootnote}{\fnsymbol{footnote}}
\footnotetext[1]{Department of Mathematics, University of Central Florida, 4393 Andromeda Loop North, Orlando, FL 32816, USA. email: abey.lopez-garcia\symbol{'100}ucf.edu.} \footnotetext[2]{Department of Mathematics and Statistics, University of South Alabama, 411 University Boulevard North, 
Mobile, AL 36688, USA. email: prokhoro\symbol{'100}southalabama.edu}

\abstract{We investigate the sequence $(P_{n}(z))_{n=0}^{\infty}$ of random polynomials generated by the three-term recurrence relation $P_{n+1}(z)=z P_{n}(z)-a_{n} P_{n-1}(z)$, $n\geq 1$, with initial conditions $P_{\ell}(z)=z^{\ell}$, $\ell=0, 1$, assuming that $(a_{n})_{n\in\mathbb{Z}}$ is a sequence of positive i.i.d. random variables. $(P_{n}(z))_{n=0}^{\infty}$ is a sequence of orthogonal polynomials on the real line, and $P_{n}$ is the characteristic polynomial of a Jacobi matrix $J_{n}$. We investigate the relation between the common distribution of the recurrence coefficients $a_{n}$ and two other distributions obtained as weak limits of the averaged  empirical and spectral measures of $J_{n}$. Our main result is a description of combinatorial relations between the moments of the aforementioned distributions in terms of certain classes of colored planar trees. Our approach is combinatorial, and the starting point of the analysis is a formula of P. Flajolet for weight polynomials associated with labelled Dyck paths.

\smallskip


\smallskip

}

\section{Introduction}

Let $\mu$ be a Borel probability measure on $(0,+\infty)$, and let $(a_{n})_{n\in\mathbb{Z}}$ be a sequence of i.i.d. random variables with distribution $\mu$, taking values in $(0,+\infty)$. We assume that all moments of $\mu$ are finite, i.e.,
\begin{equation}\label{eq:finitemoments}
m_{k}:=\int x^{k} d\mu(x)<\infty,\qquad \mbox{for all}\,\,k\in\mathbb{Z}_{\geq 0}.
\end{equation}
We emphasize that these hypotheses will be maintained throughout our work.

In this paper we consider the sequence of random polynomials $(P_{n})_{n=0}^{\infty}$ generated by the three-term recurrence relation
\begin{equation}\label{3termrec}
z P_{n}(z)=P_{n+1}(z)+a_{n} P_{n-1}(z),\qquad n\geq 1,
\end{equation}
with initial conditions
\[
P_{\ell}(z)=z^{\ell},\qquad \ell=0, 1.
\]
Note that $P_{n}(z)$ is a monic polynomial of degree $n$, and for each realization of the random variables $(a_{n})_{n\in\mathbb{Z}}$, the sequence $(P_{n}(z))_{n=0}^{\infty}$ is a sequence of orthogonal polynomials on the real line. The primary goal of this work is to study the relationship between $\mu$ and the asymptotic behavior of two discrete random measures associated with $P_{n}$ that will be defined shortly.

Let $H$ denote the tridiagonal infinite matrix
\begin{equation}\label{def:H}
H=\begin{pmatrix}
0 & 1 \\
a_{1} & 0 & 1 \\
 & a_{2} & 0 & 1 \\
 & & a_{3} & 0 & 1 \\
 & & & \ddots & \ddots & \ddots
\end{pmatrix}
\end{equation}
and let $H_{n}$ denote its principal $n\times n$ truncation
\begin{equation}\label{def:Hn}
H_{n}=\begin{pmatrix}
0 & 1 \\
a_{1} & 0 & \ddots \\
 & \ddots & \ddots & 1 \\
 & & a_{n-1} & 0
\end{pmatrix}.
\end{equation}

From the recurrence \eqref{3termrec} we easily get the relation
\[
P_{n}(z)=\det(z I_{n}-H_{n})
\]
with $I_{n}$ denoting the $n\times n$ identity matrix. Hence the zeros of $P_{n}$ are the eigenvalues of $H_{n}$. As it is well known, these eigenvalues are real and simple and will be indicated as follows:
\[
\lambda_{1}^{(n)}<\lambda_{2}^{(n)}<\ldots<\lambda_{n}^{(n)}.
\]
In this paper, we will use the notation
\begin{equation}\label{def:randomsigman}
\sigma_{n}:=\frac{1}{n}\sum_{j=1}^{n}\delta_{\lambda_{j}^{(n)}}
\end{equation}
for the empirical measure associated with $H_{n}$, where $\delta_{\lambda}$ denotes as usual the Dirac unit measure at $\lambda$. Since $H_{n}$ is diagonalizable, we have the relation
\begin{equation}\label{relsigmatrace}
\int x^{k}\,d\sigma_{n}(x)=\frac{1}{n}\,\Tr(H_{n}^{k}),\qquad k\in\mathbb{Z}_{\geq 0}.
\end{equation}

We will also analyze another random measure associated with $H_{n}$. This is the measure $\tau_{n}$ defined on $[\lambda_{1}^{(n)},\lambda_{n}^{(n)}]$ with moments given by
\begin{equation}\label{reltauH}
\int x^{k}\,d\tau_{n}(x)=\langle H_{n}^{k} e_{1}, e_{1}\rangle=H_{n}^{k}(1,1),\qquad k\in\mathbb{Z}_{\geq 0}.
\end{equation}
Because of the identity
\[
H_{n}^{k}(1,1)=J_{n}^{k}(1,1),\qquad k\in\mathbb{Z}_{\geq 0},
\]
where $J_{n}$ is the finite Jacobi matrix
\[
J_{n}=\begin{pmatrix}
0 & \sqrt{a_{1}} \\
\sqrt{a_{1}} & 0 & \ddots \\
 & \ddots & \ddots & \sqrt{a_{n-1}} \\
 & & \sqrt{a_{n-1}} & 0
\end{pmatrix},
\]
the measure $\tau_{n}$ is known as the spectral measure of $J_{n}$. A basic result in the theory of finite Jacobi matrices, which the reader can easily check, is that $\tau_{n}$ has the expression
\begin{equation}\label{def:randomtaun}
\tau_{n}=\sum_{j=1}^{n} q_{j,n}^{2}\,\delta_{\lambda_{j}^{(n)}},\qquad q_{j,n}=|v_{j}^{(n)}(1)|,
\end{equation}
where $\{v_{1}^{(n)},\ldots,v_{n}^{(n)}\}$ is any orthonormal basis in $\mathbb{R}^{n}$ formed by eigenvectors associated with the eigenvalues $\lambda_{1}^{(n)},\ldots,\lambda_{n}^{(n)}$, respectively. The first (and last) component of each eigenvector $v_{j}^{(n)}$ is non-zero, so $q_{j,n}>0$ for all $1\leq j\leq n$, and we have $\sum_{j=1}^{n}q_{j,n}^2=1$. In the theory of orthogonal polynomials on the real line, the coefficients $q_{j,n}^{2}$ are known as \emph{Christoffel numbers}. For a description of their important role in the study of orthogonal polynomials and Pad\'{e} approximation, see e.g. sections 5 and 6 of Chapter 2 in \cite{NikSor}.

In this paper we analyze the relationship between $\mu$ and two probability distributions obtained as weak limits of the average measures $\mathbb{E} \sigma_{n}$ and $\mathbb{E} \tau_{n}$. Our main result is a combinatorial description of the mutual relations between the moments of these three distributions in terms of certain classes of planar trees.  

We wish to mention some closely related works on random Jacobi matrices which partly motivated our work. Popescu \cite{Pop} studied the asymptotic distribution of eigenvalues of general Jacobi matrices assuming certain growth conditions on the off-diagonal entries, and a boundedness condition on the diagonal entries, see \cite[Theorem 1]{Pop}. Even though the situation he analyzes differs from ours, in view of \cite[Remark 2]{Pop}, it seems that after an appropriate normalization he obtains moment sequences that seem to be related as our sequences $(m_{k})_{k=0}^{\infty}$ and $(\omega_{n})_{n=0}^{\infty}$ defined in \eqref{eq:finitemoments} and \eqref{def:omegam}. If this is the case, then our Theorem~\ref{theo:invrelmomega} provides the combinatorial interpretation that he posed as an open question in his remark. Duy \cite{Khan} and Duy-Shirai \cite{KhanhShirai} studied the asymptotic behavior of spectral measures of Jacobi matrices that are obtained as tridiagonal models of Gaussian, Wishart and MANOVA beta ensembles. 

This paper is structured as follows. In Section~\ref{sec:paths} we introduce certain classes of lattice paths and associated weight polynomials. We first express the moments \eqref{relsigmatrace} and \eqref{reltauH} in terms of these weight polynomials. Then we turn our attention to the analysis of weight polynomials associated with Dyck paths and generalized Dyck paths, and analyze relations between these polynomials in terms of formal Laurent series. In Section~\ref{sec:formseries}  we introduce the two sequences $(\alpha_{n})_{n=0}^{\infty}$ and $(\omega_{n})_{n=0}^{\infty}$ that form the main object of study of our work. In Proposition~\ref{prop:analyticrel} and Theorem~\ref{theo:analyticrel} we describe some analytic relations between these sequences and the sequence $(m_{k})_{k=0}^{\infty}$. In Section~\ref{sec:trees} we introduce four classes of rooted planar trees, which are used to describe combinatorial relations between the sequences $(m_{k})_{k=0}^{\infty}$, $(\alpha_{n})_{n=0}^{\infty}$, and $(\omega_{n})_{n=0}^{\infty}$, see Theorems~\ref{theo:combmalpha} through~\ref{theo:combalphaomega}. These combinatorial relations are more direct than those described in Section~\ref{sec:formseries}, as they do not involve the intermediate quantities defined in \eqref{defin:alphank}. Finally, in Section~\ref{sec:asymp} we show that under an additional condition, the averages $\mathbb{E}\tau_{n}$ and $\mathbb{E}\sigma_{n}$ converge to symmetric probability measures whose even moments are given by the sequences $(\alpha_{n})_{n=0}^{\infty}$ and $(\omega_{n})_{n=0}^{\infty}$, respectively.

\section{Lattice paths and associated weight polynomials}\label{sec:paths}

In this section we introduce certain lattice paths and associated weight polynomials, which constitute the first ingredients in our analysis. This approach is classical for the study of orthogonal polynomials, random polynomials and random matrices. In the area of orthogonal polynomials it goes back at least to the work of Viennot \cite{Vien}, and in random matrices applications of this approach abound.  

\subsection{Path representation of quantities of interest}\label{subs:lattice}

We consider the oriented graph $\mathcal{G}=(\mathcal{V},\mathcal{E})$ with set of vertices $\mathcal{V}:=\mathbb{Z}_{\geq 0}\times \mathbb{Z}$ and set of edges
\[
\mathcal{E}:=\mathcal{E}_{u}\cup\mathcal{E}_{d},
\]
where
\begin{align*}
\mathcal{E}_{u} & :=\{(n,m)\rightarrow (n+1,m+1) : n\in\mathbb{Z}_{\geq 0}, m\in\mathbb{Z}\}\qquad (\mbox{the \emph{up steps}}),\\
\mathcal{E}_{d} & :=\{(n,m)\rightarrow (n+1,m-1) : n\in\mathbb{Z}_{\geq 0}, m\in\mathbb{Z}\}\qquad (\mbox{the \emph{down steps}}).
\end{align*}
Here, $v\rightarrow v'$ indicates the edge with initial vertex $v$ and ending vertex $v'$. By a \emph{path} on $\mathcal{G}$ we mean a finite sequence of edges
\begin{equation}\label{eq:path}
\gamma=e_{1}e_{2}\cdots e_{k},
\end{equation}
where for each $1\leq j\leq k-1,$ the ending vertex of $e_{j}$ conicides with the initial vertex of $e_{j+1}$. We say that the path in \eqref{eq:path} has \emph{length} $k$. If $(n,m)\in\mathcal{V}$ is a vertex in the path $\gamma$, we say that $\gamma$ has \emph{height} $m$ at \emph{time} $n$. We define $\max(\gamma)$ to be the maximum of all the heights of $\gamma$, and $\min(\gamma)$ to be the minimum of all the heights of $\gamma$. Also, if $q\in\mathbb{Z}$ and $\gamma$ is a path, we denote by $\gamma+q$ the path obtained by shifting $\gamma$ vertically $|q|$ units upwards or downwards according to whether $q$ is positive or negative.

Let $(a_{n})_{n\in\mathbb{Z}}$ be the sequence of random variables that we considered before. To each edge we associate a weight as follows:
\begin{equation}\label{weightedges}
w((n,m)\rightarrow (n+1,m+1))=1,\qquad w((n,m)\rightarrow (n+1,m-1))=a_{m-1}.
\end{equation}
Hence, all upstep edges have weight 1 and a downstep edge has a weight that depends on the ordinate of its ending vertex. If $\gamma$ is now a path on $\mathcal{G}$, we define its weight by
\begin{equation}\label{weightpath}
w(\gamma)=\prod_{e\subset\gamma}w(e),
\end{equation}
the product being taken over all edges of $\gamma$.

With these notions introduced, we now turn to the analysis of the traces $\Tr(H_{n}^{k})$, where $H_{n}$ is the $n\times n$ matrix in \eqref{def:Hn}. We have
\begin{equation}\label{eq:traceHn}
\Tr(H_{n}^{k})=\sum_{1\leq i_{1},\ldots,i_{k}\leq n}h_{i_{1},i_{2}} h_{i_2,i_3}\cdots h_{i_{k-1},i_{k}} h_{i_{k},i_{1}},
\end{equation}
where $h_{i,j}$, $i,j\geq 1$, is the $(i,j)$-entry of $H$. For short, we will write
\begin{equation}\label{eq:hi}
h_{\mathbf{i}}:=h_{i_{1},i_{2}} h_{i_2,i_3}\cdots h_{i_{k-1},i_{k}} h_{i_{k},i_{1}},\qquad\mathbf{i}:=(i_{1},i_{2},\ldots,i_{k},i_{1}).
\end{equation}
From the form of $H_{n}$ it follows that the only nonzero terms \eqref{eq:hi} that appear in the trace \eqref{eq:traceHn} are those associated with vectors $\mathbf{i}$ such that for each $j=1,\ldots,k,$ we have $|i_{j}-i_{j+1}|=1$ (where $i_{k+1}=i_{1}$).

With this in mind, we define $\mathcal{A}(n,k,i)$ to be the set of all vectors $\mathbf{i}=(i_{1},\ldots,i_{k+1})\in\mathbb{N}^{k+1}$ with the following properties:
\begin{itemize}
\item[1)] Each component satisfies $1\leq i_{j}\leq n$, and $i_{k+1}=i_{1}=i$.
\item[2)] For each $j=1,\ldots,k,$ we have $|i_{j}-i_{j+1}|=1$.
\end{itemize}
To each vector $\mathbf{i}\in\mathcal{A}(n,k,i)$ we associate a path $\gamma_{\mathbf{i}}=e_{1}e_{2}\cdots e_{k}$ of length $k$ on $\mathcal{G}$ having edges
\begin{equation}\label{edge:ej}
e_{j}: (j-1,i_{j})\rightarrow (j,i_{j+1}),\qquad 1\leq j\leq k.
\end{equation}
Thus, $\gamma_{\mathbf{i}}$ is a path with starting point $(0,i_{1})$ and ending point $(k,i_{1})$. It is clear that the map $\mathbf{i}\mapsto\gamma_{\mathbf{i}}$ is one-to-one, and we will denote by $\mathcal{P}(n,k,i)$ the image under this map of $\mathcal{A}(n,k,i)$. Thus, $\mathcal{P}(n,k,i)$ is the collection of all paths $\gamma$ on $\mathcal{G}$ satisfying $1\leq \min(\gamma)\leq \max(\gamma)\leq n$ and having initial point $(0,i)$ and ending point $(k,i)$.

In virtue of \eqref{eq:traceHn} and \eqref{eq:hi} we have
\begin{equation*}
\Tr(H_{n}^{k})=\sum_{i=1}^{n}\sum_{\mathbf{i}\in\mathcal{A}(n,k,i)}h_{\mathbf{i}}=\sum_{i=1}^{n}\sum_{\mathbf{i}\in\mathcal{A}(n,k,i)}w(\gamma_{\mathbf{i}})
\end{equation*}
where the second equality follows from the fact that for each $1\leq j\leq k$, the entry $h_{i_{j},i_{j+1}}$ is precisely the weight of the edge $e_{j}$ in \eqref{edge:ej}. In virtue of the bijection between $\mathcal{A}(n,k,i)$ and $\mathcal{P}(n,k,i)$, we can also write
\begin{equation}\label{eq:trace:rep}
\Tr(H_{n}^{k})=\sum_{i=1}^{n}\sum_{\gamma\in\mathcal{P}(n,k,i)}w(\gamma).
\end{equation}
Similarly, the $(1,1)$-entry of $H_{n}^{k}$ can be expressed in the form
\begin{equation}\label{eq:entry11}
H_{n}^{k}(1,1)=\sum_{\gamma\in\mathcal{P}(n,k,1)}w(\gamma).
\end{equation}

If $\gamma\in\mathcal{P}(n,k,i)$, since the path starts and ends at the same height, necessarily
\begin{equation}\label{relupdown}
\card\{\mbox{up steps in}\,\,\gamma\}=\card \{\mbox{down steps in}\,\,\gamma\},
\end{equation}
implying that $k$ is even ($k$ is the total number of edges in $\gamma$). In particular, if $k$ is odd, then $\mathcal{P}(n,k,i)=\emptyset$ for all $1\leq i\leq n$, and therefore
\begin{equation}\label{nulltrace11entry}
k\,\,\mbox{odd} \implies \Tr(H_{n}^{k})=H_{n}^{k}(1,1)=0.
\end{equation}

\subsection{Generalized Dyck paths and associated weight polynomials}

In the previous subsection we observed that the trace of $H_{n}^{k}$ is naturally associated with the collections of paths $\mathcal{P}(n,k,i)$. However, for the asymptotic analysis of these traces, it is convenient to consider a more uniform and less restricted collection of paths that we define next.

For any $n\in\mathbb{Z}_{\geq 0}$, let $\mathcal{P}_{n}$ denote the collection of all paths on $\mathcal{G}$ of length $2n$ with starting point $(0,0)$ and ending point $(2n,0)$. We will refer to $\mathcal{P}_{n}$ as the collection of \emph{generalized Dyck paths} of length $2n$ (also known as \emph{flawed Dyck paths}). Note that there are $\binom{2n}{n}$ different paths in $\mathcal{P}_{n}$.

There are two natural involutions $I:\mathcal{P}_{n}\longrightarrow\mathcal{P}_{n}$ that can be defined on $\mathcal{P}_{n}$, i.e., maps with the property that $I^{2}$ is the identity. One is the map $\gamma\mapsto\overline{\gamma}$ defined by taking the symmetric image with respect to the real axis, and the other one is the map $\gamma\mapsto\gamma^{*}$ defined by taking the symmetric image with respect to the vertical line $x=n$.

An important subset of $\mathcal{P}_{n}$ is the collection of \emph{Dyck paths} of length $2n$, which we will denote by $\mathcal{D}_{n}$. This collection consists of those paths $\gamma\in\mathcal{P}_{n}$ such that $\min(\gamma)=0$. It is well known that the cardinality of $\mathcal{D}_{n}$ is
\[
C_{n}=\frac{1}{n+1}\binom{2n}{n},
\]
the $n$th Catalan number. Observe also that $\mathcal{D}_{n}=\mathcal{D}_{n}^{*}$, i.e., $\mathcal{D}_{n}$ is invariant under the involution $\gamma\mapsto\gamma^{*}$. See Figures \ref{Dyckpath} and \ref{genDyckpath} for examples of a Dyck path and a generalized Dyck path.

To make our formulas below more symmetric, we will adopt a new notation for the random variables $a_{n}$ with $n<0$, namely we define
\begin{equation}\label{def:bn}
b_{n}:=a_{-n-1},\qquad n\in \mathbb{Z}_{\geq 0}.
\end{equation}

\begin{figure}
\begin{center}
\begin{tikzpicture}[scale=0.7]
\draw[line width=1.5pt]  (-3.03,0) -- (11.5,0);
\draw[line width=1.5pt]  (-3,0) -- (-3,3.5);
\draw[line width=0.7pt] (-3,0) -- (-2,1) -- (-1,0) -- (0,1) -- (1,2) -- (2,1) -- (3,2) -- (4,3) -- (5,2) -- (6,1) -- (7,0) -- (8,1) -- (9,2) -- (10,1) -- (11,0);
\draw [line width=0.5] (-2,0) -- (-2,-0.12);
\draw [line width=0.5] (-1,0) -- (-1,-0.12);
\draw [line width=0.5] (0,0) -- (0,-0.12);
\draw [line width=0.5] (1,0) -- (1,-0.12);
\draw [line width=0.5] (2,0) -- (2,-0.12);
\draw [line width=0.5] (3,0) -- (3,-0.12);
\draw [line width=0.5] (4,0) -- (4,-0.12);
\draw [line width=0.5] (5,0) -- (5,-0.12);
\draw [line width=0.5] (6,0) -- (6,-0.12);
\draw [line width=0.5] (7,0) -- (7,-0.12);
\draw [line width=0.5] (8,0) -- (8,-0.12);
\draw [line width=0.5] (9,0) -- (9,-0.12);
\draw [line width=0.5] (10,0) -- (10,-0.12);
\draw [line width=0.5] (11,0) -- (11,-0.12);
\draw [line width=0.5] (-3,1) -- (-3.12,1);
\draw [line width=0.5] (-3,2) -- (-3.12,2);
\draw [line width=0.5] (-3,3) -- (-3.12,3);
\draw [dotted] (-3,1) -- (11.5,1);
\draw [dotted] (-3,2) -- (11.5,2);
\draw [dotted] (-3,3) -- (11.5,3);
\draw (-2,-0.1) node[below, scale=0.8]{$1$};
\draw (-1,-0.1) node[below, scale=0.8]{$2$};
\draw (0,-0.1) node[below, scale=0.8]{$3$};
\draw (1,-0.1) node[below, scale=0.8]{$4$};
\draw (2,-0.1) node[below, scale=0.8]{$5$};
\draw (3,-0.1) node[below, scale=0.8]{$6$};
\draw (4,-0.1) node[below, scale=0.8]{$7$};
\draw (5,-0.1) node[below, scale=0.8]{$8$};
\draw (6,-0.1) node[below, scale=0.8]{$9$};
\draw (7,-0.1) node[below, scale=0.8]{$10$};
\draw (8,-0.1) node[below, scale=0.8]{$11$};
\draw (9,-0.1) node[below, scale=0.8]{$12$};
\draw (10,-0.1) node[below, scale=0.8]{$13$};
\draw (11,-0.1) node[below, scale=0.8]{$14$};
\draw (-3.1,1) node[left, scale=0.8]{$1$};
\draw (-3.1,2) node[left, scale=0.8]{$2$};
\draw (-3.1,3) node[left, scale=0.8]{$3$};
\end{tikzpicture}
\end{center}
\caption{Example of a Dyck path of length $14$ with weight $a_{0}^{3} a_{1}^3 a_{2}$.}
\label{Dyckpath}
\end{figure}
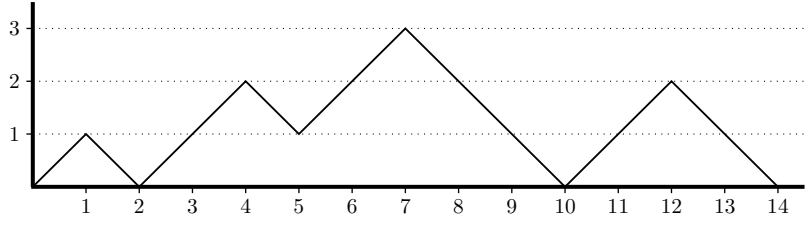

\begin{figure}
\begin{center}
\begin{tikzpicture}[scale=0.7]
\draw[line width=1.5pt]  (-3,0) -- (11.5,0);
\draw[line width=1.5pt]  (-3,-2.5) -- (-3,2.5);
\draw[line width=0.7pt] (-3,0) -- (-2,-1) -- (-1,0) -- (0,1) -- (1,2) -- (2,1) -- (3,0) -- (4,-1) -- (5,-2) -- (6,-1) -- (7,0) -- (8,-1) -- (9,0) -- (10,1) -- (11,0);
\draw [line width=0.5] (-2,0) -- (-2,-0.12);
\draw [line width=0.5] (-1,0) -- (-1,-0.12);
\draw [line width=0.5] (0,0) -- (0,-0.12);
\draw [line width=0.5] (1,0) -- (1,-0.12);
\draw [line width=0.5] (2,0) -- (2,-0.12);
\draw [line width=0.5] (3,0) -- (3,-0.12);
\draw [line width=0.5] (4,0) -- (4,-0.12);
\draw [line width=0.5] (5,0) -- (5,-0.12);
\draw [line width=0.5] (6,0) -- (6,-0.12);
\draw [line width=0.5] (7,0) -- (7,-0.12);
\draw [line width=0.5] (8,0) -- (8,-0.12);
\draw [line width=0.5] (9,0) -- (9,-0.12);
\draw [line width=0.5] (10,0) -- (10,-0.12);
\draw [line width=0.5] (11,0) -- (11,-0.12);
\draw [line width=0.5] (-3,1) -- (-3.12,1);
\draw [line width=0.5] (-3,2) -- (-3.12,2);
\draw [dotted] (-3,1) -- (11.5,1);
\draw [dotted] (-3,2) -- (11.5,2);
\draw [dotted] (-3,-1) -- (11.5,-1);
\draw [dotted] (-3,-2) -- (11.5,-2);
\draw (-2,-0.1) node[below, scale=0.8]{$1$};
\draw (-1,-0.1) node[below, scale=0.8]{$2$};
\draw (0,-0.1) node[below, scale=0.8]{$3$};
\draw (1,-0.1) node[below, scale=0.8]{$4$};
\draw (2,-0.1) node[below, scale=0.8]{$5$};
\draw (3,-0.1) node[below, scale=0.8]{$6$};
\draw (4,-0.1) node[below, scale=0.8]{$7$};
\draw (5,-0.1) node[below, scale=0.8]{$8$};
\draw (6,-0.1) node[below, scale=0.8]{$9$};
\draw (7,-0.1) node[below, scale=0.8]{$10$};
\draw (8,-0.1) node[below, scale=0.8]{$11$};
\draw (9,-0.1) node[below, scale=0.8]{$12$};
\draw (10,-0.1) node[below, scale=0.8]{$13$};
\draw (11,-0.1) node[below, scale=0.8]{$14$};
\draw (-3.1,1) node[left, scale=0.8]{$1$};
\draw (-3.1,2) node[left, scale=0.8]{$2$};
\draw (-3.1,0) node[left, scale=0.8]{$0$};
\draw (-3.1,-1) node[left, scale=0.8]{$-1$};
\draw (-3.1,-2) node[left, scale=0.8]{$-2$};
\end{tikzpicture}
\end{center}
\caption{Example of a generalized Dyck path of length $14$ with weight $a_{0}^{2} a_{1} b_{0}^3 b_{1}$.}
\label{genDyckpath}
\end{figure}
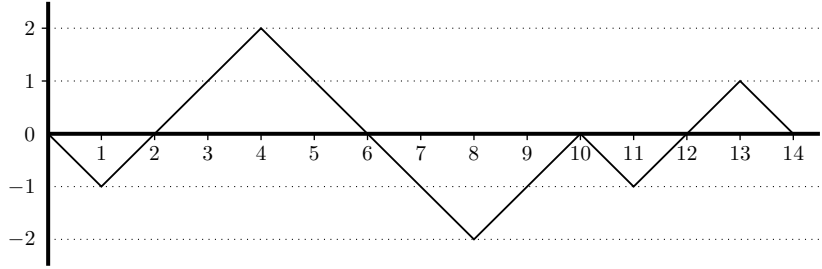

We now introduce three sequences of polynomials $(W_{n})_{n=0}^{\infty}$, $(A_{n})_{n=0}^{\infty}$ and $(B_{n})_{n=0}^{\infty}$, defined as the weight polynomials associated with the collections of paths $\mathcal{P}_{n}$, $\mathcal{D}_{n}$ and $\overline{\mathcal{D}}_{n}$, respectively, where $\overline{\mathcal{D}}_{n}$ is the image of $\mathcal{D}_{n}$ under the map $\gamma\mapsto\overline{\gamma}$. Precisely, for each $n\in\mathbb{Z}_{\geq 0}$, we let
\begin{align}
W_{n} & =\sum_{\gamma\in\mathcal{P}_{n}} w(\gamma),\label{def:Wm}\\
A_{n} & =\sum_{\gamma\in\mathcal{D}_{n}} w(\gamma),\label{def:Am}\\
B_{n} & =\sum_{\gamma\in\overline{\mathcal{D}}_{n}} w(\gamma),\label{def:Bm}
\end{align}
where by definition $W_{0}=A_{0}=B_{0}=1$. In general, if $\mathcal{S}\subset \mathcal{P}_{n}$, then we call the expression
\[
\sum_{\gamma\in\mathcal{S}}w(\gamma)
\]
the weight polynomial associated with $\mathcal{S}$.

Observe that $W_{n}$ is a polynomial in the $2n$ variables $a_{0},\ldots,a_{n-1}$ and $b_{0},\ldots, b_{n-1}$, while $A_{n}$ and $B_{n}$ are polynomials in the $n$ variables $a_{0},\ldots,a_{n-1}$ and $b_{0},\ldots,b_{n-1}$, respectively. To emphasize this dependence, sometimes we write
\begin{align*}
W_{n} & =W_{n}(a_{0},\ldots,a_{n-1};b_{0},\ldots,b_{n-1}),\\
A_{n} & =A_{n}(a_{0},\ldots,a_{n-1}), \\
B_{n} & =B_{n}(b_{0},\ldots,b_{n-1}).
\end{align*}
The explicit expressions of some of these polynomials are given below:
\begin{align*}
A_{0} & =1 \\
A_{1} & = a_{0}\\
A_{2} & = a_{0} (a_{0}+a_{1})\\
A_{3} & = a_{0} (a_{0}^{2}+2 a_{0} a_{1}+a_{1}^{2}+a_{1} a_{2})\\
W_{0} & = 1\\
W_{1} & = a_{0}+b_{0} \\
W_{2} & = a_{0} (a_{0}+a_{1})+2 a_{0} b_{0}+b_{0} (b_{0}+b_{1}) \\
W_{3} & =a_{0} (a_{0}^{2}+2 a_{0} a_{1}+a_{1}^{2}+a_{1} a_{2})+a_{0} b_{0}(3 a_{0}+3 b_{0}+2 a_{1}+2 b_{1})+b_{0} (b_{0}^{2}+2 b_{0} b_{1}+b_{1}^{2}+b_{1} b_{2}).
\end{align*}

In the following proposition we gather some elementary properties of these weight polynomials:

\begin{proposition}
The following properties hold for every $n\in\mathbb{Z}_{\geq 0}$:
\begin{itemize}
\item[1)] The polynomials $W_{n}$, $A_{n}$, $B_{n}$ are homogeneous polynomials of degree $n$.
\item[2)] We have the symmetry property
\[
W_{n}(a_{0},\ldots,a_{n-1};b_{0},\ldots,b_{n-1})=W_{n}(b_{0},\ldots,b_{n-1};a_{0},\ldots,a_{n-1}).
\]
\item[3)] We have
\[
B_{n}=A_{n}(b_{0},\ldots,b_{n-1}).
\]
\end{itemize}
\end{proposition}
\begin{proof}
The weight of any path $\gamma\in\mathcal{P}_{n}$ is the product of $n$ variables in $\{a_{i}, b_{i}\}_{i=0}^{n-1}$ since the path contains exactly $n$ down steps. This proves the first statement.

It is easy to see that for any $\gamma\in\mathcal{P}_{n}$, $w(\overline{\gamma})$ is obtained from $w(\gamma)$ by replacing any variable $a_{i}$ by $b_{i}$, and any variable $b_{i}$ by $a_{i}$. Therefore
\[
W_{n}(a_{0},\ldots,a_{n-1};b_{0},\ldots,b_{n-1})
=\sum_{\gamma\in\mathcal{P}_{n}} w(\gamma)=\sum_{\gamma\in\mathcal{P}_{n}} w(\overline{\gamma})=W_{n}(b_{0}\ldots,b_{n-1};a_{0},\ldots,a_{n-1}).
\]
The third statement is obvious.
\end{proof}

We also need to define the following polynomials obtained from \eqref{def:Am}--\eqref{def:Bm} by shifting the variables. For each $k\geq 0$, let
\begin{align}
A_{n}^{(k)} & :=A_{n}(a_{k},\ldots,a_{k+n-1})\qquad n\geq 0, \label{def:Amshifted} \\
B_{n}^{(k)} & :=B_{n}(b_{k},\ldots,b_{k+n-1})\qquad n\geq 0, \label{def:Bmshifted}
\end{align}
in particular $A_{n}=A_{n}^{(0)}$, $B_{n}=B_{n}^{(0)}$ for all $n$.

In this paper we will frequently use formal Laurent series in the space $\mathbb{C}((z^{-1}))$, equipped with the usual addition and multiplication of series, which makes this space a field. We define the following series:
\begin{align}
W(z) & :=\sum_{n=0}^{\infty} \frac{W_{n}}{z^{2n+1}},\label{def:seriesW}\\
A^{(k)}(z) & :=\sum_{n=0}^{\infty}\frac{A_{n}^{(k)}}{z^{2n+1}},\qquad k\geq 0,\label{def:seriesAk}\\
B^{(k)}(z) & :=\sum_{n=0}^{\infty}\frac{B_{n}^{(k)}}{z^{2n+1}},\qquad k\geq 0. \label{def:seriesBk}
\end{align}
In the case $k=0$, we use the notation
\begin{align}
A(z) & :=\sum_{n=0}^{\infty}\frac{A_{n}}{z^{2n+1}},\label{def:seriesA}\\
B(z) & :=\sum_{n=0}^{\infty}\frac{B_{n}}{z^{2n+1}}.\label{def:seriesB}
\end{align}

The decoupling formula that we present in \eqref{eq:relWAB} is well-known in the theory of deterministic bi-infinite Jacobi matrices and is due to Masson-Repka \cite{MasRep}, who gave an analytic proof of it in their paper. We give here a combinatorial proof of this formula based on lattice paths.

\begin{proposition}
The following relation holds:
\begin{equation}\label{eq:relWAB}
W(z)=\frac{1}{z-a_{0} A^{(1)}(z)-b_{0} B^{(1)}(z)}.
\end{equation}
We also have that for each $k\geq 0$,
\begin{align}
A^{(k)}(z) & =\frac{1}{z-a_{k} A^{(k+1)}(z)}, \label{eq:relAA} \\
B^{(k)}(z) & =\frac{1}{z-b_{k} B^{(k+1)}(z)}. \label{eq:relBB}
\end{align}
\end{proposition}
\begin{proof}
First, consider the set $\widehat{\mathcal{P}}_{n}\subset\mathcal{P}_{n}$ consisting of all paths $\gamma$ that do not touch the real line except at times $0$ and $2n$. The portion of a path $\gamma\in\widehat{\mathcal{P}}_{n}$ that corresponds to the time interval $[1,2n-1]$ is either above the line $y=1$ or below the line $y=-1$, and therefore it can be identified with the translation of a path in $\mathcal{D}_{n-1}$ or in $\overline{\mathcal{D}}_{n-1}$, respectively. This identification shows that the weight polynomial of $\widehat{\mathcal{P}}_{n}$ is exactly the polynomial $a_{0} A^{(1)}_{n-1}+b_{0} B^{(1)}_{n-1}$.

We can partition the set $\mathcal{P}_{n}$ into different subsets by looking at the first time that a path returns to zero. For each $k=1,\ldots,n$, we define $\mathcal{S}_{k}\subset\mathcal{P}_{n}$ to be the set of all paths that return first to zero at time $2k$. Then the sets $\mathcal{S}_{k}$ indeed form a partition of $\mathcal{P}_{n}$, and any path in $\mathcal{S}_{k}$ is a concatenation of a path in $\widehat{\mathcal{P}}_{k}$ with a horizontal translation of a path in $\mathcal{P}_{n-k}$. Hence the weight polynomial of $\mathcal{S}_{k}$ is $(a_{0} A^{(1)}_{k-1}+b_{0} B^{(1)}_{k-1}) W_{n-k}$ and adding we obtain the relation
\[
W_{n}=\sum_{k=1}^{n}(a_{0} A^{(1)}_{k-1}+b_{0} B^{(1)}_{k-1}) W_{n-k},\qquad n\geq 1.
\]
In terms of series this means
\[
z W(z)-1= (a_{0} A^{(1)}(z)+b_{0} B^{(1)}(z))\,W(z),
\]
which gives \eqref{eq:relWAB}.

If we use the same argument above, replacing $\mathcal{P}_{n}$ by $\mathcal{D}_{n}$ and $\overline{\mathcal{D}}_{n}$, we will obtain \eqref{eq:relAA}--\eqref{eq:relBB} in the case $k=0$. Since the polynomials $A^{(k)}$ and $B^{(k)}$ all have the same structure, it is obvious that the same relation will also hold for any $k\geq 1$.
\end{proof}

\begin{corollary}
We have the relation
\begin{equation}\label{eq:relWAB:2}
W(z)=\frac{1}{A(z)^{-1}+B(z)^{-1}-z}.
\end{equation}
For any $k\geq 1$, we have
\begin{equation}\label{eq:contfrac:1}
A(z)=\cfrac{1}{z-\cfrac{a_{0}}{z-\cfrac{a_{1}}{z-\raisebox{-0.5\height}{$\ddots -\cfrac{a_{k-1}}{z-a_{k} A^{(k+1)}(z)}$}}}},
\end{equation}
and a similar identity holds between the $B$ polynomials.
\end{corollary}
\begin{proof}
The relations follow immediately from \eqref{eq:relWAB}--\eqref{eq:relBB}.
\end{proof}

In view of \eqref{eq:relWAB:2} and \eqref{eq:contfrac:1}, the series $A(z)$ and $W(z)$ are naturally connected to certain random continued fractions:
\begin{align*}
A(z) & \rightarrow \cfrac{1}{z-\cfrac{a_{0}}{z-\cfrac{a_{1}}{z-\cfrac{a_{2}}{\ddots}}}}\\
W(z) & \rightarrow \cfrac{1}{z-\cfrac{a_{0}}{z-\cfrac{a_{1}}{z-\cfrac{a_{2}}{\ddots}}}-\cfrac{b_{0}}{z-\cfrac{b_{1}}{z-\cfrac{b_{2}}{\ddots}}}}.
\end{align*}

\subsection{Explicit formulae for the weight polynomials}

In this subsection we give some explicit formulae for the weight polynomials defined in \eqref{def:Wm}--\eqref{def:Bm}. We first introduce some definitions and notations.

Given an integer $n\geq 1$, let
\begin{equation}\label{def:partition:1}
\mathrm{C}(n):=\{(n_{0},\ldots,n_{r}): n_{0}+\cdots+n_{r}=n, n_{j}\in\mathbb{N}\,\,\mbox{for all}\,\,0\leq j\leq r\}
\end{equation}
denote the set of all integer compositions of $n$. For example,
\[
\mathrm{C}(4)=\{(4), (3,1), (1,3), (2,2), (2,1,1), (1,2,1), (1,1,2), (1,1,1,1)\}.
\]
It is also convenient to define the set
\begin{equation}\label{def:partition:2}
\mathrm{C}(0):=\{e\}
\end{equation}
whose only element is the empty sequence $e$.

The following formulas for the polynomials $A_{n}$ and $B_{n}$ are due to Flajolet \cite[Prop. 3B]{Flaj}: For any $n\in\mathbb{Z}_{\geq 0}$,
\begin{align}
A_{n} & =\sum_{(n_{0},\ldots,n_{r})\in\mathrm{C}(n)}\binom{n_{0}+n_{1}-1}{n_{0}-1}\binom{n_{1}+n_{2}-1}{n_{1}-1}\cdots \binom{n_{r-1}+n_{r}-1}{n_{r-1}-1}\,a_{0}^{n_{0}} a_{1}^{n_{1}}\cdots a_{r}^{n_{r}},\label{eq:Flaj:An}\\
B_{n} & =\sum_{(n_{0},\ldots,n_{r})\in\mathrm{C}(n)}\binom{n_{0}+n_{1}-1}{n_{0}-1}\binom{n_{1}+n_{2}-1}{n_{1}-1}\cdots \binom{n_{r-1}+n_{r}-1}{n_{r-1}-1}\,b_{0}^{n_{0}} b_{1}^{n_{1}}\cdots b_{r}^{n_{r}}.\label{eq:Flaj:Bn}
\end{align}
In the case $n=0$, we understand the right-hand sides of \eqref{eq:Flaj:An}--\eqref{eq:Flaj:Bn} to be $1$, and if $n\geq 1$, $r=0$, the product of the binomials to be $1$ as well. In \cite{Flaj}, Flajolet gives the polynomials $A_{n}$ the name \textit{Stieltjes-Rogers polynomials}.  

In order to make our formulas below more compact and manageable, we introduce some more definitions and notations. Given $n\in\mathbb{Z}_{\geq 0}$ and a composition $\ov{n}\in \mathrm{C}(n)$, we define
\begin{equation}\label{def:rho:1}
\rho_{1}(\ov{n}):=\begin{cases}
\prod_{j=0}^{r-1}\binom{n_{j}+n_{j+1}-1}{n_{j}-1} & \qquad \mbox{if}\,\,\ov{n}=(n_{0},\ldots,n_{r}),\,\,r\geq 1,\\
1 & \qquad \mbox{if}\,\,\ov{n}=(n),\,\,n\geq 1,\,\,\mbox{or}\,\,\ov{n}=e,
\end{cases}
\end{equation}
where $e$ is the element in \eqref{def:partition:2}.
We also define
\begin{equation}\label{def:func:a}
a(\ov{n}):=\begin{cases}
\prod_{j=0}^{r}a_{j}^{n_{j}} & \qquad \mbox{if}\,\,\ov{n}=(n_{0},\ldots,n_{r}),\,\,r\geq 0,\\
1 & \qquad \mbox{if}\,\,\ov{n}=e,
\end{cases}
\end{equation}
and similarly
\begin{equation}\label{def:func:b}
b(\ov{n}):=\begin{cases}
\prod_{j=0}^{r}b_{j}^{n_{j}} & \qquad \mbox{if}\,\,\ov{n}=(n_{0},\ldots,n_{r}),\,\,r\geq 0,\\
1 & \qquad \mbox{if}\,\,\ov{n}=e.
\end{cases}
\end{equation}
With these definitions, \eqref{eq:Flaj:An}--\eqref{eq:Flaj:Bn} take the form
\begin{align}
A_{n} & =\sum_{\ov{n}\in\mathrm{C}(n)}\rho_1(\ov{n})\,a(\ov{n}),\label{eq:Flaj:comp:An}\\
B_{n} & =\sum_{\ov{n}\in\mathrm{C}(n)}\rho_1(\ov{n})\,b(\ov{n}).\label{eq:Flaj:comp:Bn}
\end{align}

Our main goal in this subsection is to prove a formula for the polynomials $W_{n}$ that is analogous to \eqref{eq:Flaj:comp:An}--\eqref{eq:Flaj:comp:An}. To accomplish this we need some more definitions.

Given $n\in\mathbb{Z}_{\geq 0}$, let
\begin{equation}\label{def:genCn}
\widehat{\mathrm{C}}(n):=\bigcup_{j=0}^{n} \mathrm{C}(j)\times\mathrm{C}(n-j),
\end{equation}
i.e., $\widehat{\mathrm{C}}(n)$ consists of all pairs $(\ov{p}, \ov{q})$ with $\ov{p}\in\mathrm{C}(j)$ and $\ov{q}\in\mathrm{C}(n-j)$ for some $0\leq j\leq n$. Additionally, for a given $(\ov{p},\ov{q})\in\widehat{\mathrm{C}}(n)$ we define
\begin{equation}\label{def:genrho}
\rho_{2}((\overline{p},\overline{q})):=\begin{cases}
\binom{n_{0}+n_{0}'}{n_{0}}\,\rho_{1}(\overline{p})\,\rho_{1}(\overline{q}) & \qquad \mbox{if}\,\,\ov{p}=(n_{0},\ldots,n_{r}),\,\,r\geq 0,\,\,\ov{q}=(n_{0}',\ldots,n_{s}'),\,\,s\geq 0,\\[0.5em]
\rho_{1}(\overline{p}) & \qquad \mbox{if}\,\,\overline{q}=e,\\[0.5em]
\rho_{1}(\overline{q}) & \qquad \mbox{if}\,\,\overline{p}=e.
\end{cases}
\end{equation}
Using \eqref{def:genCn}--\eqref{def:genrho}, we give in \eqref{formula:Wm} below a formula for $W_{n}$.

In the sequel, we use the notation
\[
X(z):=A^{(1)}(z),\qquad Y(z):=B^{(1)}(z),
\]
see \eqref{def:seriesAk}--\eqref{def:seriesBk}. Then, in virtue of \eqref{eq:relWAB}--\eqref{eq:relBB}, we have
\begin{align}
W(z) & =\frac{1}{z-(a_{0} X(z)+b_{0} Y(z))},\label{eq:relWXY}\\
A(z) & =\frac{1}{z-a_{0} X(z)},\label{eq:relAX}\\
B(z) & =\frac{1}{z-b_{0} Y(z)}.\label{eq:relBY}
\end{align}

Finally, given a formal Laurent series $S(z)\in\mathbb{C}((z^{-1}))$, the notation $[S]_{n}$ will indicate the coefficient of $z^{-n}$ in the expression of $S(z)$. If $S(z)$ is not the zero series, we define
\begin{equation}\label{def:degreeseries}
\deg(S):=\max\{n\in\mathbb{Z}: [S]_{-n}\neq 0\}.
\end{equation}
Note that this agrees with the definition of degree of a polynomial in $z$.

\begin{proposition}
For any $n\geq 0$,
\begin{equation}
\begin{aligned}
A_{n} & =\sum_{k=0}^{n} a_{0}^{k}\, [X^{k}]_{2n-k},\\
B_{n} & =\sum_{k=0}^{n} b_{0}^{k}\, [Y^{k}]_{2n-k}.
\end{aligned}
\label{eq:polyABXY}
\end{equation}
Moreover, for every $k=0,\ldots,n$,
\begin{align}
[X^{k}]_{2n-k} & =\sum_{(n_{1},\ldots,n_{r})\in\mathrm{C}(n-k)}\binom{k+n_{1}-1}{k-1}\binom{n_{1}+n_{2}-1}{n_{1}-1}\cdots\binom{n_{r-1}+n_{r}-1}{n_{r-1}-1} a_{1}^{n_{1}}\cdots a_{r}^{n_{r}},\label{rep:powersX}\\
[Y^{k}]_{2n-k} & =\sum_{(n_{1},\ldots,n_{r})\in\mathrm{C}(n-k)}\binom{k+n_{1}-1}{k-1}\binom{n_{1}+n_{2}-1}{n_{1}-1}\cdots\binom{n_{r-1}+n_{r}-1}{n_{r-1}-1} b_{1}^{n_{1}}\cdots b_{r}^{n_{r}},\label{rep:powersY}
\end{align}
where $\mathrm{C}(n-k)$ is defined in \eqref{def:partition:1}--\eqref{def:partition:2}. The right-hand side in \eqref{rep:powersX}--\eqref{rep:powersY} is understood to be $1$ in the case $n=k$. For each $k=0,\ldots,n,$ the expression $a_{0}^{k}[X^{k}]_{2n-k}$ is the weight polynomial associated with the set of all Dyck paths of length $2n$ with exactly $k$ returns to zero.
\end{proposition}
\begin{proof}
The two formulas in \eqref{eq:polyABXY} follow immediately from Lemma~\ref{lemma:app:1} and the two relations \eqref{eq:relAX}--\eqref{eq:relBY}.

To justify \eqref{rep:powersX}, observe that the first relation in \eqref{eq:polyABXY} gives a decomposition of $A_{n}$ into $n+1$ different groups of terms, where the group $a_{0}^{k} [X^{k}]_{2n-k}$ gathers all terms in the expression of $A_{n}$ whose $a_{0}$ factor is exactly $a_{0}^{k}$. Therefore, the expression for $a_{0}^{k} [X^{k}]_{2n-k}$ is obtained by taking $n_{0}=k$ in \eqref{eq:Flaj:An}. This gives immediately \eqref{rep:powersX} after dividing by $a_{0}^{k}$. The proof for \eqref{rep:powersY} is obviously the same. The last statement in the Lemma also follows immediately.
\end{proof}

In the proof of \eqref{formula:Wm} below, we will apply \eqref{rep:powersX}--\eqref{rep:powersY} in the case when $k=0, n>0$. In this case, we obviously have $[X^{k}]_{2n-k}=[Y^{k}]_{2n-k}=0$, which we get through the formulas \eqref{rep:powersX}--\eqref{rep:powersY} if we adopt, as in \cite{Flaj}, the convention
\begin{equation}\label{convention}
\binom{n}{-1}=\delta_{n,-1}
\end{equation}
where $\delta$ is Kronecker's symbol. Indeed, if $k=0$, $n>0$, with this convention we get $\binom{k+n_{1}-1}{k-1}=\binom{n_{1}-1}{-1}=0$ since $n_{1}\geq 1$.

\begin{theorem}
The following formula holds for every $n\geq 0$:
\begin{equation}\label{formula:Wm}
W_{n}=\sum_{(\ov{p},\ov{q})\in\widehat{\mathrm{C}}(n)}\rho_{2}((\ov{p},\ov{q}))\,a(\ov{p})\,b(\ov{q}),
\end{equation}
see \eqref{def:genCn}--\eqref{def:genrho} and \eqref{def:func:a}--\eqref{def:func:b} for the meaning of notation.
\end{theorem}
\begin{proof}
First, from \eqref{eq:relWXY} and Lemma~\ref{lemma:app:1}, we obtain that for every $n\geq 0$,
\begin{equation}\label{firstexpWm}
W_{n}=\sum_{m=0}^{n}[(a_{0} X+b_{0} Y)^{m}]_{2n-m}
=\sum_{m=0}^{n}\sum_{k=0}^{m}\binom{m}{k} a_{0}^{k}\,b_{0}^{m-k} \, [X^{k}\,Y^{m-k}]_{2n-m}.
\end{equation}
In order to make the formula symmetric, we rewrite the above expression in the form
\begin{equation}\label{eq:decompWm:1}
W_{n}=\sum_{m=0}^{n}\sum_{n_{0}+n_{0}'=m}
\binom{n_{0}+n_{0}'}{n_{0}} a_{0}^{n_{0}}\,b_{0}^{n_{0}'} \, [X^{n_{0}}\,Y^{n_{0}'}]_{2n-m},
\end{equation}
where the second summation is taken over the set
\[
\{(n_{0},n_{0}')\in\mathbb{Z}_{\geq 0}^{2}: n_{0}+n_{0}'=m\}.
\]
Now we identify the coefficient $[X^{n_{0}}\,Y^{n_{0}'}]_{2n-m}$ in the series expansion of $X^{n_{0}}\,Y^{n_{0}'}$. Writing
\[
X^{n_{0}}(z)\,Y^{n_{0}'}(z)=\sum_{i=0}^{\infty}\frac{[X^{n_{0}}]_{n_{0}+2i}}{z^{n_{0}+2i}}\,\sum_{j=0}^{\infty}\frac{[Y^{n_{0}'}]_{n_{0}'+2j}}{z^{n_{0}'+2j}},
\]
we deduce that
\begin{equation}\label{eq:decompWm:2}
[X^{n_{0}}\,Y^{n_{0}'}]_{2n-m}=\sum_{j=0}^{n-m}[X^{n_{0}}]_{n_{0}+2j}\,[Y^{n_{0}'}]_{2n-m-2j-n_{0}}.
\end{equation}
Applying \eqref{rep:powersX}--\eqref{rep:powersY},
\begin{align}
[X^{n_{0}}]_{n_{0}+2j} & =\sum_{(n_{1},\ldots,n_{r})\in\mathrm{C}(j)}\binom{n_{0}+n_{1}-1}{n_{0}-1}\cdots\binom{n_{r-1}+n_{r}-1}{n_{r-1}-1} a_{1}^{n_{1}}\cdots a_{r}^{n_{r}},\label{eq:decompWm:3}\\
[Y^{n_{0}'}]_{2n-m-2j-n_{0}} & =\sum_{(n_{1}',\ldots,n_{s}')\in\mathrm{C}(n-m-j)}\binom{n_{0}'+n_{1}'-1}{n_{0}'-1}\cdots\binom{n_{s-1}'+n_{s}'-1}{n_{s-1}'-1} b_{1}^{n_{1}'}\cdots b_{s}^{n_{s}'}.\label{eq:decompWm:4}
\end{align}

Combining \eqref{eq:decompWm:1}--\eqref{eq:decompWm:4} we obtain
\begin{equation}\label{formula:Wm:alt}
W_{n}=\sum_{m=0}^{n}\sum_{n_{0}+n_{0}'=m}\sum_{j=0}^{n-m}\sum_{\substack{(n_{1},\ldots,n_{r})\in\mathrm{C}(j)\\(n_{1}',\ldots,n_{s}')\in\mathrm{C}(n-m-j)}}\pi(n_{0},\ldots,n_{r};n_{0}',\ldots,n_{s}')\,a_{0}^{n_{0}}\cdots a_{r}^{n_{r}} b_{0}^{n_{0}'}\cdots b_{s}^{n_{s}'}
\end{equation}
where
\begin{equation}\label{def:kappa}
\pi(n_{0},\ldots,n_{r};n_{0}',\ldots,n_{s}')=\binom{n_{0}+n_{0}'}{n_{0}} \times \prod_{k=0}^{r-1}\binom{n_{k}+n_{k+1}-1}{n_{k}-1} \times \prod_{k=0}^{s-1}\binom{n_{k}'+n_{k+1}'-1}{n_{k}'-1}.
\end{equation}
The expression in \eqref{formula:Wm:alt} reduces to the expression in \eqref{formula:Wm}. Indeed, observe first that some terms in \eqref{formula:Wm:alt} give a null contribution. These are the terms obtained by taking $n_{0}=0$ and $j>0$, or taking $n_{0}'=0$ and $j<n-m$. In both cases the expression \eqref{def:kappa} is zero as a consequence of \eqref{convention}. In the remaining cases, if we construct the vectors $\ov{p}=(n_{0},\ldots,n_{r})$ and $\ov{q}=(n_{0}',\ldots,n_{s}')$, then clearly $(\ov{p},\ov{q})\in\widehat{\mathrm{C}}(n)$ (if one of the vectors is the zero vector, we identify it with the element $e$ in $\mathrm{C}(0)$) and we have
\[
\pi(n_{0},\ldots,n_{r};n_{0}',\ldots,n_{s}')\,a_{0}^{n_{0}}\cdots a_{r}^{n_{r}}\, b_{0}^{n_{0}'}\cdots b_{s}^{n_{s}'}=\rho_{2}((\ov{p},\ov{q}))\,a(\ov{p})\,b(\ov{q}).
\]
It is also clear that there is a one-to-one correspondance between the terms in \eqref{formula:Wm:alt} that give a non-zero contribution and the set of all pairs $(\ov{p},\ov{q})\in\widehat{\mathrm{C}}(n)$. With this we conclude the proof.
\end{proof}

Below we present alternative formulae for the polynomials $A_{n}$ and $W_{n}$. These formulae are, however, not convenient for our purposes, and they will not be used in the rest of the paper, but we include them for their independent interest. Formula \eqref{eq:altformAm} for $A_{n}$ appears in Touchard \cite{Tou}. It is also a particular case of what Aptekarev-Kaliaguine-Van Iseghem \cite[sections 1 and 3]{AptKalIs} call \emph{genetic sums}. For the sake of completeness, we give an independent proof of \eqref{eq:altformAm}. Note also that, in contrast to \eqref{formula:Wm}, the formula given below for $W_{n}$ is asymmetric, in the sense that we write $W_{n}$ as a polynomial in the variables $\{a_{j}\}_{j=-n}^{n-1}$, instead of the variables $\{a_{j}, b_{j}\}_{j=0}^{n-1}$.

\begin{proposition}
The following formulas hold for every $n\geq 0$,
\begin{align}
A_{n} & = \sum_{i_{1}=0}^{0}\,\sum_{i_{2}=0}^{i_{1}+1}\,\sum_{i_{3}=0}^{i_{2}+1}
\cdots\sum_{i_{n}=0}^{i_{n-1}+1}\,\prod_{j=1}^{n} a_{i_{j}},\label{eq:altformAm}\\
W_{n} & =\sum_{i_{1}=-1}^{n-1}\,\sum_{i_{2}=i_{1}-1}^{n-2}\,\sum_{i_{3}=i_{2}-1}^{n-3}
\cdots\sum_{i_{n}=i_{n-1}-1}^{0}\,\prod_{j=1}^{n} a_{i_{j}}.\label{eq:altformWm}
\end{align}
\end{proposition}
\begin{proof}
Given a general path $\gamma\in\mathcal{P}_{n}$, we denote by $d_{j}$, $j=1,\ldots,n,$ the $j$th down step edge of $\gamma$, counting from left to right.

We first prove \eqref{eq:altformWm}. For each $j=1,\ldots,n$, let $a_{i_j}$ be the weight of the edge $d_{j}$ in the general path $\gamma$, i.e., $a_{i_j}=w(d_{j})$, cf. \eqref{weightedges}. We first look at the possible values for $a_{i_1}=w(d_{1})$. These values are clearly those with index in the range $-1\leq i_{1}\leq n-1$, and observe that the number of up step edges that precede $d_{1}$ is $i_{1}+1$. Indeed, $i_{1}=-1$ if $d_{1}$ is the first edge in $\gamma$, $i_{1}=0$ if $d_{1}$ is the second edge in $\gamma$, and so on, up to the value $i_{1}=n-1$ in case that the first $n$ edges in $\gamma$ are all up step edges. For a fixed value of $i_{1}$, the number of remaining up step edges that follow $d_{1}$ is $n-i_{1}-1$. Hence the possible values for $a_{i_{2}}=w(d_{2})$ are those with index in the range $i_{1}-1\leq i_{2}\leq n-2$, where the upper bound is obtained by adding $n-i_{1}-1$, the number of remaining up step edges, to $i_{1}-1$. For a fixed value of $i_{2}$ in that range, the total number of up step edges that precede $d_{2}$ is $(i_{1}+1)+(i_{2}-i_{1}+1)=i_{2}+2$, so the total number of up step edges that follow $d_{2}$ is $n-i_{2}-2$. In general, given $2 \leq j \leq n$, suppose that one of the possible values for $i_{j-1}$ is fixed, and the number of up step edges that follow $d_{j-1}$ is $n-i_{j-1}-j+1$. Then clearly the possible values for $i_{j}$ are in the range $i_{j-1}-1\leq i_{j}\leq n-j$, and the number of up step edges that follow $d_{j}$ is $n-i_{j}-j$. This shows that the collection of all possible values for $w(\gamma)$, $\gamma\in\mathcal{P}_{n}$, is exactly given by \eqref{eq:altformWm}.

The proof of \eqref{eq:altformAm} follows the same reasoning, with the difference that the analysis is done backwards instead of forward. Now we let $a_{i_{j}}:=w(d_{n-j+1})$, $1\leq j\leq m$, so $a_{i_{1}}$ is the weight of the last down step $d_{n}$, $a_{i_{2}}$ is the weight of $d_{n-1}$, and so on. If $\gamma\in\mathcal{D}_{n}$ is a Dyck path, then necessarily its last edge is the down step edge $d_{n}$ that joins the points $(2n-1,1)$ and $(2n,0)$, hence $a_{i_{1}}=a_0$. The possible values for $a_{i_{2}}$ are $a_{0}$ (if there is exactly one up step edge between $d_{n-1}$ and $d_{n}$) and $a_{1}$ (if there is no up step edge between $d_{n-1}$ and $d_{n}$), i.e., $0\leq i_{2}\leq 1$. So the number of up step edges between $d_{n-1}$ and $d_{n}$ is $1-i_{2}$. For a fixed value of $i_{2}$, the possible values for $i_{3}$ are in the range $0\leq i_{3}\leq i_{2}+1$, since between $d_{n-2}$ and $d_{n-1}$ one can have at most $i_{2}+1$ up step edges. For a fixed value of $i_{3}$, the number of up step edges between $d_{n-2}$ and $d_{n-1}$ is $i_{2}-i_{3}+1$, so the total number of up step edges that precede $d_{n-2}$ is $n-(i_{2}-i_{3}+1)-(1-i_{2})=n+i_{3}-2$. In general, by induction one can show that for each $2\leq j\leq n$, the possible values for $i_{j}$ are $0\leq i_{j}\leq i_{j-1}+1$ and the total number of up step edges that precede $d_{n-j+1}$ is $n+i_{j}-j+1$. In the final case $j=n$, this shows that the number of up steps that precede $d_{1}$ is precisely $i_{n}+1$. This concludes the proof of \eqref{eq:altformAm}.
\end{proof}

\section{The sequences $(\alpha_{n})_{n=0}^{\infty}$, $(\omega_{n})_{n=0}^{\infty}$, and associated formal series}\label{sec:formseries}

In this section we describe some properties of the two sequences $(\alpha_{n})_{n=0}^{\infty}$ and $(\omega_{n})_{n=0}^{\infty}$ defined by
\begin{align}
\alpha_{n} & :=\mathbb{E}(A_{n}),\qquad n\geq 0,\label{def:alpham}\\
\omega_{n} & :=\mathbb{E}(W_{n}),\qquad n\geq 0.\label{def:omegam}
\end{align}
Note that in virtue of \eqref{eq:finitemoments} and the definition of the polynomials $A_{n}$ and $W_{n}$, the values $\alpha_{n}$ and $\omega_{n}$ are finite for all $n$. The first few values are
\begin{align*}
\alpha_{0} & =1\\
\alpha_{1} & =m_{1}\\
\alpha_{2} & =m_{2}+m_{1}^{2}\\
\alpha_{3} & =m_{3}+3 m_{2} m_{1}+m_{1}^{3}\\
\alpha_{4} & =m_{4}+4 m_{3} m_{1}+3 m_{2}^2+5 m_{2} m_{1}^{2}+m_{1}^4\\
\alpha_{5} & =m_{5}+5 m_{4} m_{1}+10 m_{3} m_{2}+7 m_{3} m_{1}^2+11 m_{2}^2 m_{1}+7 m_{2} m_{1}^3+m_{1}^5\\
\omega_{0} & =1\\
\omega_{1} & = 2m_{1}\\
\omega_{2} & = 2 m_{2}+4 m_{1}^{2}\\
\omega_{3} & = 2 m_{3}+12 m_{2} m_{1}+6 m_{1}^{3}\\
\omega_{4} & = 2 m_{4}+16 m_{3} m_{1}+12 m_{2}^2+32 m_{2} m_{1}^2+8 m_{1}^4\\
\omega_{5} & = 2 m_{5}+20 m_{4} m_{1}+40 m_{3} m_{2}+50 m_{3} m_{1}^2+70 m_{2}^2 m_{1}+60 m_{2} m_{1}^3+10 m_{1}^5.
\end{align*}

In this section we also define the quantities
\begin{equation}\label{defin:alphank}
\alpha_{n}^{(k)}:=\mathbb{E}([A^{k}]_{k+2n}),\qquad k, n\in\mathbb{Z}_{\geq 0},
\end{equation}
where $A$ is the series in \eqref{def:seriesA}. These values are all finite. Observe that $\alpha_{0}^{(k)}=1$ for all $k\geq 0$ and $\alpha_{n}^{(0)}=0$ for all $n>0$. Also note that $\alpha_{n}^{(1)}=\alpha_{n}$ for all $n\geq 0$.

We also introduce the following formal Laurent series:
\begin{align}
g_{k}(z) & :=\sum_{n=0}^{\infty}\frac{\alpha_{n}^{(k)}}{z^{2n+k}},\qquad k\geq 0,\label{def:seriesgk}\\
f(z) & :=\sum_{n=0}^{\infty}\frac{\omega_{n}}{z^{2n+1}},\label{def:seriesf}
\end{align}
and note that $g_{0}\equiv 1$. Finally, in analogy to \eqref{def:func:a}--\eqref{def:func:b} we define, for each $\ov{n}\in \mathrm{C}(n)$, $n\geq 0$, the function
\begin{equation}\label{def:mfunc}
m(\ov{n}):=\begin{cases}
\prod_{j=0}^{r}m_{n_{j}} & \qquad \mbox{if}\,\,\ov{n}=(n_{0},\ldots,n_{r}),\,\,r\geq 0,\\
1 & \qquad \mbox{if}\,\,\ov{n}=e.
\end{cases}
\end{equation}

\begin{proposition}\label{prop:analyticrel}
The following identities hold. For every $n\in\mathbb{Z}_{\geq 0}$,
\begin{align}
\alpha_{n} & =\sum_{\ov{n}\in\mathrm{C}(n)}\rho_1(\ov{n})\,m(\ov{n}),\label{eq:formalpha}\\
\omega_{n} & =\sum_{(\ov{p},\ov{q})\in\widehat{\mathrm{C}}(n)}\rho_{2}(\ov{p},\ov{q})\,m(\ov{p})\,m(\ov{q}).\label{eq:formomega}
\end{align}
For all $k, n\in\mathbb{Z}_{\geq 0}$,
\begin{equation}\label{eq:formalphamk}
\alpha_{n}^{(k)} =\sum_{\ov{n}\in\mathrm{C}(n)}\binom{\ov{n}(1)+k-1}{k-1}\,\rho_{1}(\ov{n})\,m(\ov{n}),
\end{equation}
where $\ov{n}(1)$ denotes the first entry of $\ov{n}$.
\end{proposition}
\begin{proof}
The formulas are obtained immediately by taking the expected value in \eqref{eq:Flaj:An}, \eqref{formula:Wm} and \eqref{rep:powersX}, and using the fact that the random variables $\{a_{n}, b_{n}\}_{n=0}^{\infty}$ are i.i.d. with moments \eqref{eq:finitemoments}. Also note that $\mathbb{E}([X^{k}]_{k+2n})=\mathbb{E}([A^{k}]_{k+2n})$ for all $k, n\geq 0$.
\end{proof}

Before we state our next result, we make a clarification regarding the expressions on the right-hand sides of \eqref{eq:relgkgk}--\eqref{eq:relfgk}. In general, the addition of infinitely many series in the space $\mathbb{C}((z^{-1}))$ is not well-defined. However, if we have a sequence $(\xi_{n}(z))_{n=0}^{\infty}$ of series such that $\deg(\xi_{n})\longrightarrow -\infty$ as $n\rightarrow\infty$ (cf. \eqref{def:degreeseries}), then the expression $\xi(z)=\sum_{n=0}^{\infty}\xi_{n}(z)$ is well-defined as the series whose coefficients are $[\xi]_{j}=\sum_{n=0}^{\infty}[\xi_{n}]_{j}$, $j\in\mathbb{Z}$, since the latter summation is in fact finite for every $j$. This is the case of the expressions on the right-hand sides of \eqref{eq:relgkgk}--\eqref{eq:relfgk}.

\begin{theorem}\label{theo:analyticrel}
The following relations hold. For any $n\geq 0$,
\begin{align}
\alpha_{n}^{(k)} & =\sum_{j=0}^{n}\binom{j+k-1}{k-1}\,m_{j}\,\alpha_{n-j}^{(j)},\qquad k\geq 0,\label{eq:relalphas}\\
\omega_{n} & =\sum_{j=0}^{n}\sum_{\ell=0}^{n-j} m_{j}\,\alpha_{\ell}^{(j)}\,\alpha_{n-j-\ell}^{(j+1)},\label{eq:relomegaalpha:simple}\\
\omega_{n} & =\sum_{j=0}^{n}\sum_{i=0}^{j}\sum_{\ell=0}^{n-j}\binom{j}{i}\, m_{i}\,m_{j-i}\,\,\alpha_{\ell}^{(i)}\,\alpha_{n-j-\ell}^{(j-i)}.\label{eq:relomegaalpha}
\end{align}
We also have the following relations:
\begin{align}
g_{k}(z) & =\sum_{j=0}^{\infty}\binom{j+k-1}{k-1} \frac{m_{j}\,g_{j}(z)}{z^{j+k}},\qquad k\geq 0,\label{eq:relgkgk}\\
f(z) & =\sum_{j=0}^{\infty} m_{j}\,g_{j}(z)\,g_{j+1}(z), \label{eq:relfgk:simple}\\
f(z) & =\sum_{j=0}^{\infty}\sum_{i=0}^{j}\binom{j}{i}\frac{m_{i}\,m_{j-i}\,g_{i}(z)\,g_{j-i}(z)}{z^{j+1}}.\label{eq:relfgk}
\end{align}
\end{theorem}
\begin{proof}
From the relations \eqref{eq:relAX} and \eqref{coeffpowerR}, where in the latter we take $R=A$ and $S=a_{0} X$, we obtain
\[
[A^{k}]_{k+2n}=\sum_{j=0}^{n}\binom{j+k-1}{k-1}\,a_{0}^{j}\,[X^{j}]_{2n-j}.
\]
Taking expectation, and using the facts that the random variables $a_{0}$ and $[X^{j}]_{\ell}$ are independent and $\mathbb{E}([X^{k}]_{k+2m})=\mathbb{E}([A^{k}]_{k+2m})$, we obtain \eqref{eq:relalphas}. The identity \eqref{eq:relomegaalpha} is obtained by taking expectation in \eqref{firstexpWm} and in \eqref{eq:decompWm:2}, and taking into account that the random variables $a_{0}$, $b_{0}$, $[X^{j}]_{\ell}$ and $[Y^{r}]_{s}$ are independent. To justify \eqref{eq:relomegaalpha:simple}, we write 
\begin{align*}
\omega_{n} & =\sum_{j=0}^{n}\sum_{i=0}^{j}\sum_{\ell=0}^{n-j}
\binom{j}{i}\,m_{i}\,m_{j-i}\,\alpha_{\ell}^{(i)}\,\alpha_{n-j-\ell}^{(j-i)}\\
&=\sum_{i=0}^{n}\sum_{\ell=0}^{n-i}\sum_{j=i}^{n-\ell}\binom{j}{i}\,m_{i}\,m_{j-i}\,\alpha_{\ell}^{(i)}\,\alpha_{n-j-\ell}^{(j-i)}\\
&=\sum_{i=0}^{n}\sum_{\ell=0}^{n-i}m_{i}\,\alpha_{\ell}^{(i)}\left(\sum_{j=i}^{n-\ell}\binom{j}{i}\,m_{j-i}\,\alpha_{n-j-\ell}^{(j-i)}\right)
\end{align*} 
and in virtue of \eqref{eq:relalphas}, the expression inside the parenthesis is exactly $\alpha_{n-i-\ell}^{(i+1)}$.

The identities \eqref{eq:relgkgk}--\eqref{eq:relfgk} follow immediately from \eqref{eq:relalphas}--\eqref{eq:relomegaalpha} and \eqref{def:seriesgk}--\eqref{def:seriesf}. In the case of \eqref{eq:relgkgk}, if we call $h_{k}(z)$ the right-hand side of \eqref{eq:relgkgk}, then it is clear that $\deg(h_{k})=-k$, and for every $n\geq 0$,
\[
[h_{k}]_{k+2n}=\sum_{j=0}^{\infty}\binom{j+k-1}{k-1}\,m_{j}\,\left[\frac{g_{j}(z)}{z^{j+k}}\right]_{k+2n},
\]
where, in view of \eqref{def:seriesgk},
\[
\left[\frac{g_{j}(z)}{z^{j+k}}\right]_{k+2n}=\left[\sum_{\ell=0}^{\infty}\frac{\alpha_{\ell}^{(j)}}{z^{2(j+\ell)+k}}\right]_{k+2n}=\begin{cases}
\alpha_{n-j}^{(j)}, & \mbox{if}\,\,j\leq n,\\
0, & \mbox{if}\,\,j>n.
\end{cases}
\]
Hence
\[
[h_{k}]_{k+2n}=\sum_{j=0}^{n}\binom{j+k-1}{k-1}\,m_{j}\,\alpha_{n-j}^{(j)}.
\]
On the other hand, for every $n\geq 0$,
\[
\left[\frac{g_{j}(z)}{z^{j+k}}\right]_{k+2n+1}=0,
\]
so $[h_{k}]_{k+2n+1}=0$. Consequently, in virtue of \eqref{def:seriesgk} and \eqref{eq:relalphas} we obtain the desired identity $h_{k}=g_{k}$.

The proofs of \eqref{eq:relfgk:simple} and \eqref{eq:relfgk} are done similarly, using \eqref{eq:relomegaalpha:simple} and \eqref{eq:relomegaalpha}. We leave them to the reader.
\end{proof}

\section{Trees and relations}\label{sec:trees}

In this section we describe combinatorial relations between the three sequences $(m_{n})_{n=0}^{\infty}$, $(\alpha_{n})_{n=0}^{\infty}$, and $(\omega_{n})_{n=0}^{\infty}$, complementing the formulas \eqref{eq:formalpha} and \eqref{eq:formomega}.

We first describe how the quantities $m_{n}$ and $\omega_{n}$ can be expressed in terms of the quantities $\alpha_{k}$, $k=0,\ldots,n$. The first few relations of this type are
\begin{align*}
m_{0} & = \alpha_{0}\\
m_{1} & =\alpha_{1}\\
m_{2} & =\alpha_{2}-\alpha_{1}^2\\
m_{3} & = \alpha_{3}-3\alpha_{2}\alpha_{1}+2\alpha_{1}^{3}\\
m_{4} & =\alpha_{4} - 4\alpha_{3}\alpha_{1} + 13\alpha_{2}\alpha_{1}^2 - 3\alpha_{2}^2 - 7\alpha_{1}^4\\
m_{5} & =\alpha_{5} - 5 \alpha_{4} \alpha_1 - 10 \alpha_3\alpha_2+23 \alpha_3 \alpha_1^2+34 \alpha_2^2 \alpha_1 - 79 \alpha_2 \alpha_1^3+36 \alpha_{1}^5\\
\omega_{0} & =\alpha_{0}\\
\omega_{1} & =2\alpha_{1}\\
\omega_{2} & =2\alpha_{2}+2\alpha_{1}^2\\
\omega_{3} & =2\alpha_{3}+6 \alpha_{2} \alpha_{1}-2\alpha_{1}^{3}\\
\omega_{4} & =2\alpha_{4}+8\alpha_{3} \alpha_{1}-14 \alpha_{2} \alpha_{1}^2+6 \alpha_{2}^2+6 \alpha_{1}^4\\
\omega_{5} & =2\alpha_{5}+10\alpha_{4} \alpha_{1}+20 \alpha_{3} \alpha_{2}-24 \alpha_{3} \alpha_{1}^2-42 \alpha_{2}^2 \alpha_{1}+72 \alpha_{2} \alpha_{1}^3-28 \alpha_{1}^{5}.
\end{align*}

Our formulas will use certain classes of planar trees. The first one of these classes is described next. Let $n\geq 1$ be an integer and let $\ov{n}=(n_{0},n_{1},\ldots,n_{r})\in \mathrm{C}(n)$ be fixed. We consider \emph{rooted leveled} trees associated with $\ov{n}$, defined by the following conditions:

\begin{itemize}
\item[T1)] Each vertex of the tree is represented by a positive integer, called the \emph{value} of the vertex. The tree has a root vertex with value $n$.

\item[T2)] The vertices of the tree are distributed in $d+1$ disjoint levels $\ell=0,\ldots,d$, $d\geq 0$, where level $0$ is formed solely by the root vertex, and level $d$ consists of the vertices with values $n_{0}, n_{1}, \ldots, n_{r}$, from left to right. The vertices at level $\ell$ are those at a distance $\ell$ from the root. For each $\ell=0,\ldots,d-1,$ we represent graphically the vertices at level $\ell$ above the vertices at level $\ell+1$.

\item[T3)] For each $\ell=0,\ldots,d-1,$ every vertex at level $\ell$ is a neighbor of at least one vertex at level $\ell+1$, and there exists at least one vertex at level $\ell$ that is a neighbor of at least two vertices at level $\ell+1$. The vertices at level $\ell+1$ that are neighbors of a vertex $v$ at level $\ell$ are called the \emph{direct descendants} of $v$, and $v$ is called the \emph{parent} of these vertices. For each $\ell=0,\ldots,d,$ the sum of the values of the vertices at level $\ell$ is $n$. If $v_{1},\ldots,v_{k}$ are the direct descendants of a vertex $v$, then the sum of the values of $v_{1},\ldots,v_{k}$ is the value of $v$.

\item[T4)] For each $\ell=0\,\ldots, d-1,$ if a vertex $v$ at level $\ell$ has only one direct descendant $v'$ at level $\ell+1$, then $v'$ has only one direct descendant as well, unless $v'$ is a vertex in the last level $d$ of the tree.
\end{itemize}

A tree satisfying the four properties $T1)$ -- $T4)$ is called an \emph{admissible} tree associated with $\ov{n}$, see Fig. \ref{Admtree}. The collection of all these trees is denoted $\mathcal{T}_{1}(\ov{n})$. We say that an admissible tree has \emph{height} $d$ if it has $d+1$ levels. If a vertex $v$ of an admissible tree has more than one direct descendant, we say that $v$ is a \emph{multi-branching} vertex.

\begin{figure}
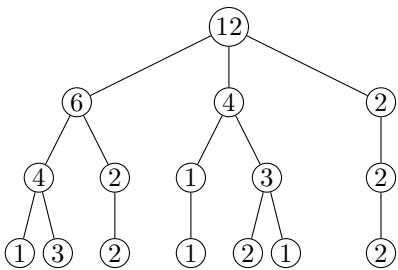

\begin{center}
\begin{tikz}
[level distance=10mm,
every node/.style={inner sep=1pt},
level 1/.style={sibling distance=20mm},
level 2/.style={sibling distance=10mm},
level 3/.style={sibling distance=5mm}]
\node[circle,draw]{$12$}
child {node[circle,draw]{$6$}
child {node[circle,draw]{$4$}
child {node[circle,draw]{$1$}}
child {node[circle,draw]{$3$}}
}
child {node[circle,draw]{$2$}
child {node[circle,draw]{$2$}}
}
}
child {node[circle,draw]{$4$}
child {node[circle,draw]{$1$}
child {node[circle,draw]{$1$}}
}
child {node[circle,draw] {$3$}
child {node[circle,draw]{$2$}}
child {node[circle,draw]{$1$}}
}
}
child {node[circle,draw]{$2$}
child {node[circle,draw]{$2$}
child {node[circle,draw]{$2$}
}
}
};
\end{tikz}
\end{center}
\caption{Example of an admissible tree with height $3$ and associated with $(1,3,2,1,2,1,2)\in\mathrm{C}(12)$.}
\label{Admtree}
\end{figure}

We define now a weight for each admissible tree. First, for any vertex $v$ of an admissible tree, let
\begin{equation}\label{def:kappaweight}
\kappa_{1}(v):=\begin{cases}
-\rho_{1}((\lambda_1,\ldots,\lambda_{s})) & \parbox[t]{.55\textwidth}{if $v$ is multi-branching, and $\lambda_{1},\ldots,\lambda_{s},$ $s\geq 2$, are the values of the direct descendants of $v$, from left to right,}\\[1.5em]
1 & \mbox{otherwise},
\end{cases}
\end{equation}
recall that $\rho_{1}((\lambda_1,\ldots,\lambda_{s}))=\prod_{j=1}^{s-1}\binom{\lambda_{j}+\lambda_{j+1}-1}{\lambda_{j}-1}$. Then, for an admissible tree $t$ we define
\begin{equation}\label{def:weighttree}
w_{1}(t):=\prod_{v}\kappa_{1}(v)
\end{equation}
where the product is taken over all vertices of $t$.

In analogy to \eqref{def:mfunc}, we define the following expressions for each composition $\ov{n}\in\mathrm{C}(n)$, $n\geq 0$:
\begin{align}
\alpha(\ov{n}) & :=\begin{cases}\prod_{j=0}^{r} \alpha_{n_{j}} & \qquad \mbox{if}\,\,\ov{n}=(n_{0},\ldots,n_{r}), r\geq 0,\\
1 & \qquad \mbox{if}\,\,\ov{n}=e,
\end{cases}\label{def:alphafunc} \\
\omega(\ov{n}) & :=\begin{cases}\prod_{j=0}^{r} \omega_{n_{j}} & \qquad \mbox{if}\,\,\ov{n}=(n_{0},\ldots,n_{r}), r\geq 0,\\
1 & \qquad \mbox{if}\,\,\ov{n}=e.
\end{cases}\label{def:omegafunc}
\end{align}

Before we state our first result in this section, we want to define an operation on admissible trees called \emph{extension}. Let $t$ be an admissible tree. We say that a tree $s$ is an extension of $t$ by $k$ units if $s$ is obtained by appending to each vertex $v$ in the last level of $t$ a vertical line tree with $l$ edges and vertices with the same value as the vertex $v$, see Fig.~\ref{Exampleext}. If $l=0$, we understand $s=t$. However, if $l\geq 1$, the new tree $s$ is not an admissible tree (it does not satisfy property T3)).

\begin{figure}
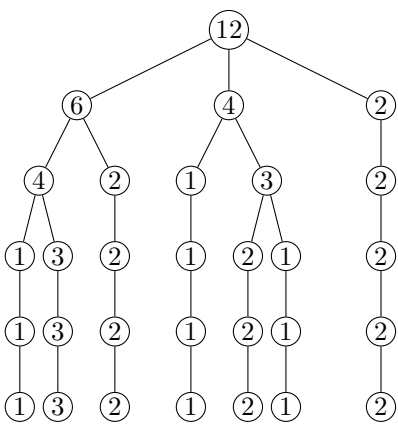

\begin{center}
\begin{tikz}
[level distance=10mm,
every node/.style={inner sep=1pt},
level 1/.style={sibling distance=20mm},
level 2/.style={sibling distance=10mm},
level 3/.style={sibling distance=5mm}]
\node[circle,draw]{$12$}
child {node[circle,draw]{$6$}
child {node[circle,draw]{$4$}
child {node[circle,draw]{$1$}
child {node[circle,draw]{$1$}
child {node[circle,draw]{$1$}}}}
child {node[circle,draw]{$3$}
child {node[circle,draw]{$3$}
child {node[circle,draw]{$3$}}
}
}
}
child {node[circle,draw]{$2$}
child {node[circle,draw]{$2$}
child {node[circle,draw]{$2$}
child {node[circle,draw]{$2$}
}
}
}
}
}
child {node[circle,draw]{$4$}
child {node[circle,draw]{$1$}
child {node[circle,draw]{$1$}
child {node[circle,draw]{$1$}
child {node[circle,draw]{$1$}
}
}
}
}
child {node[circle,draw] {$3$}
child {node[circle,draw]{$2$}
child {node[circle,draw]{$2$}
child {node[circle,draw]{$2$}
}
}
}
child {node[circle,draw]{$1$}
child {node[circle,draw]{$1$}
child {node[circle,draw]{$1$}
}
}
}
}
}
child {node[circle,draw]{$2$}
child {node[circle,draw]{$2$}
child {node[circle,draw]{$2$}
child {node[circle,draw]{$2$}
child {node[circle,draw]{$2$}
}
}
}
}
};
\end{tikz}
\end{center}
\caption{The tree shown is an extension of the tree in Fig. \ref{Admtree} by two units.}
\label{Exampleext}
\end{figure}

\begin{theorem}\label{theo:combmalpha}
For each integer $n\geq 1$,
\begin{equation}\label{formula:inv:1}
m_{n}=\sum_{\ov{n}\in \mathrm{C}(n)} \phi_{1}(\ov{n})\,\alpha(\ov{n}),
\end{equation}
where
\begin{equation}\label{def:funcphi}
\phi_{1}(\ov{n}):=\sum_{t\in\mathcal{T}_{1}(\ov{n})}w_{1}(t).
\end{equation}
Moreover, for each $n\geq 2$ we have
\begin{equation}\label{eq:sum:inv:1}
\sum_{\ov{n}\in \mathrm{C}(n)}\phi_{1}(\ov{n})=0.
\end{equation}
See \eqref{def:alphafunc} and \eqref{def:weighttree} for the meaning of notation.
\end{theorem}
\begin{proof}
The proof is by induction on $n$. The result holds trivially in the case $n=1$ since $\mathrm{C}(1)=\{(1)\}$, $\phi_{1}((1))=1$ and $m_{1}=\alpha_{1}$.

Suppose that \eqref{formula:inv:1} is valid for all values $n=1,\ldots,k-1,$ $k\geq 2$, and let us show that it is also valid for $n=k$. First, we deduce from \eqref{eq:formalpha} that
\begin{equation}\label{eq:iteration}
m_{k}=\alpha_{k}-\sum_{\ov{n}\in\mathrm{C}_{*}(k)} \rho_{1}(\ov{n})\,m(\ov{n}),
\end{equation}
where
\begin{equation}\label{def:Cstar}
\mathrm{C}_{*}(k):=\mathrm{C}(k)\setminus\{(k)\}.
\end{equation}
Now we rewrite \eqref{eq:iteration} in a convenient way.

Fix a particular $\ov{n}=(n_{0},\ldots,n_{r})\in\mathrm{C}_{*}(k)$. Then $m(\ov{n})=m_{n_0}\cdots m_{n_{r}}$ and we can apply the induction hypothesis to each $m_{n_{j}}$ since $1\leq n_{j}\leq k-1$ for all $0\leq j\leq r$. So
\[
m_{n_{j}}=\sum_{\ov{p}_{j}\in\mathrm{C}(n_{j})} \phi_{1}(\ov{p}_{j})\,\alpha(\ov{p}_{j}),\qquad 0\leq j\leq r,
\]
and therefore
\begin{equation}\label{eq:decomp1}
m(\ov{n})=\prod_{j=0}^{r}\,\,\sum_{\ov{p}_{j}\in\mathrm{C}(n_{j})}\phi_{1}(\ov{p}_{j})\,\alpha(\ov{p}_{j})=\sum_{(\ov{p}_{0},\ldots,\ov{p}_{r})}\prod_{j=0}^{r}\phi_{1}(\ov{p}_{j})\,\alpha(\ov{p}_{j}),
\end{equation}
where the last summation runs over all tuples $(\ov{p}_{0},\ldots,\ov{p}_{r})\in \mathrm{C}(n_{0})\times\cdots\times \mathrm{C}(n_{r})$. Applying \eqref{def:funcphi} we get
\begin{equation}\label{eq:decomp2}
\prod_{j=0}^{r}\phi_{1}(\ov{p}_{j})\,\alpha(\ov{p}_{j})=\prod_{j=0}^{r}\left(\sum_{t_{j}\in\mathcal{T}_{1}(\ov{p}_{j})}w_{1}(t_{j})\right)\,\alpha(\ov{p}_{j})
=\sum_{(t_{0},\ldots,t_{r})}\prod_{j=0}^{r} w_{1}(t_{j})\,\alpha(\ov{p}_{j}),
\end{equation}
where the last summation runs over all tuples $(t_{0},\ldots,t_{r})\in \mathcal{T}_{1}(\ov{p}_{0})\times \cdots \times \mathcal{T}_{1}(\ov{p}_{r})$.

From \eqref{eq:iteration}, \eqref{eq:decomp1} and \eqref{eq:decomp2} we deduce that
\begin{equation}\label{eq:mn:summation}
m_{k}=\alpha_{k}-\sum_{(n_{0},\ldots,n_{r})\in\mathrm{C}_{*}(k)} \sum_{(\ov{p}_{0},\ldots,\ov{p}_{r})}\sum_{(t_{0},\ldots,t_{r})}\rho_{1}((n_{0},\ldots,n_{r}))\,\prod_{j=0}^{r} w_{1}(t_{j})\,\alpha(\ov{p}_{j}).
\end{equation}
The summation in \eqref{eq:mn:summation} has one term for each choice of $(n_{0},\ldots,n_{r})\in\mathrm{C}_{*}(k)$ and corresponding choices of $(\ov{p}_{0},\ldots,\ov{p}_{r})\in \mathrm{C}(n_{0})\times\cdots\times \mathrm{C}(n_{r})$ and $(t_{0},\ldots,t_{r})\in \mathcal{T}_{1}(\ov{p}_{0})\times \cdots \times \mathcal{T}_{1}(\ov{p}_{r})$. The idea now is to show that for each term in this summation we can make the identification
\begin{equation}\label{eq:rel:weights}
-\rho_{1}((n_{0},\ldots,n_{r}))\,\prod_{j=0}^{r}w_{1}(t_{j})\,\alpha(\ov{p}_{j})=w_{1}(t)\,\alpha(\ov{p})
\end{equation}
for a certain admissible tree $t$ associated with a vector $\ov{p}\in\mathrm{C}_{*}(k)$ (the tree and the vector depends of course on the particular term).

So let us define two sets and a map between these sets to describe this identification. The first set $\mathcal{S}_{1}$ consists of all tuples $(\ov{n}, \ov{\pi}, \ov{t})$ satisfying the conditions
\begin{align*}
\ov{n} & =(n_{0},\ldots,n_{r})\in\mathrm{C}_{*}(k),\\
\ov{\pi} & =(\ov{p}_{0},\ldots,\ov{p}_{r})\in\mathrm{C}(n_{0})\times\cdots\times \mathrm{C}(n_{r}),\\
\ov{t} & =(t_{0},\ldots,t_{r})\in\mathcal{T}_{1}(\ov{p}_{0})\times\cdots\times \mathcal{T}_{1}(\ov{p}_{r}).
\end{align*}
The second set $\mathcal{T}_{1}$ is the collection of all admissible trees associated with a vector $\ov{p}\in\mathrm{C}_{*}(k)$. Let $T_{1}:\mathcal{S}_{1}\longrightarrow\mathcal{T}_{1}$ be the map that assigns to each $(\ov{n}, \ov{\pi}, \ov{t})\in\mathcal{S}_{1}$ the tree $t\in\mathcal{T}_{1}$ constructed using the following multi-step procedure:
\begin{itemize}
\item[1)] Construct the admissible tree $s$ with two levels (levels $0$ and $1$) associated with the vector $\ov{n}=(n_{0},\ldots,n_{r})$.
\item[2)] Let $d_{j}$ be the height of the admissible tree $t_{j}$ in $\ov{t}=(t_{0},\ldots,t_{r})$, and let $d:=\max_{0\leq j\leq r} d_{j}$. For each $j=0,\ldots,r,$ construct a new tree $\widehat{t}_{j}$ by performing an extension of $t_{j}$ with $d-d_{j}$ units. Note that at least one $t_{j}$ remains unchanged after performing the extensions.
\item[3)] For each $j=0,\ldots,r,$ append the tree $\widehat{t}_{j}$ to the tree $s$ by using the vertex in level $1$ of $s$ with value $n_{j}$ as the root vertex of the tree $\widehat{t}_{j}$. Let $t$ be the resulting tree after completing this process.
\end{itemize}

It is evident that $t$ is an admissible tree associated with a vector $\ov{p}\in\mathrm{C}_{*}(k)$. Moreover, from \eqref{def:kappaweight} and \eqref{def:weighttree} we easily deduce that \eqref{eq:rel:weights} indeed holds. Also, the reader can easily check that the map $T_{1}$ is a one-to-one and onto. Therefore, from \eqref{eq:mn:summation} and \eqref{eq:rel:weights} we deduce the desired identity \eqref{formula:inv:1} for $n=k$.

Finally, we prove \eqref{eq:sum:inv:1}. If we formally set $\alpha_{n}=1$ for all $n\geq 0$, then \eqref{formula:inv:1} transforms into $m_{n}=\sum_{\ov{n}\in\mathrm{C}(n)}\phi_{1}(\ov{n})$ for all $n\geq 1$. Hence \eqref{eq:sum:inv:1} will be justified if we show that
\begin{equation}\label{eq:partalpham}
\alpha_{n}=1 \,\,\mbox{for all}\,\,n\geq 0 \,\,\implies m_{n}=0\,\,\mbox{for all}\,\,n\geq 2.
\end{equation}
We prove this implication by induction on $n$. First, assuming the hypothesis we get $m_{1}=\alpha_{1}=1$ and $m_{2}=\alpha_{2}-\alpha_{1}^{2}=0$. Let $k\geq 3$ and assume that $m_{n}=0$ for all $2\leq n\leq k-1$. Then applying \eqref{eq:iteration} we get
\[
m_{k}=\alpha_{k}-\sum_{\ov{n}\in\mathrm{C}_{*}(k)}\rho_{1}(\ov{n})\,m(\ov{n})=\alpha_{k}-\rho_{1}((1,\ldots,1))\,m_{1}^{k}=1-1=0,
\]
since the only non-zero term in the summation is the one corresponding to the constant vector $\ov{n}=(1,\ldots,1)$.
\end{proof}

In order to state our next result, we need to define a second class of trees. Let $n\geq 1$ be an integer and let $\ov{n}=(n_{0},\ldots,n_{r})\in \mathrm{C}(n)$ be fixed. We say that a tree $t$ belongs to the class $\mathcal{T}_{2}(\ov{n})$, if in addition to $T1)$--$T4)$, $t$ satisfies the following \emph{edge coloring} conditions:

\begin{itemize}
\item[$A1)$] Every edge of the tree is either a \textit{blue edge} ($b$-edge) or a \textit{red edge} ($r$-edge).

\item[$A2)$] If $t$ does not reduce to a single vertex, $v$ is the root vertex of $t$, and $v_{1},\ldots,v_{s},$ $s\geq 2$, are the direct descendants of $v$, listed from left to right, then there exists $1\leq j\leq s-1$ such that the edges connecting $v$ with $v_{1},\ldots,v_{j}$ are all $b$-edges, and the edges connecting $v$ with $v_{j+1},\ldots,v_{s}$ are all $r$-edges.

\item[$A3)$] If $v$ is a vertex in $t$ different from the root vertex, then the edge connecting $v$ with its parent has the same color as the edges connecting $v$ with its direct descendants, in case $v$ has any.\end{itemize}

Note that among the edges connecting the root vertex with its direct descendants, there are at least one $b$-edge and and at least one $r$-edge. Let $t\in\mathcal{T}_{2}(\overline{n})$. For any vertex $v$ of $t$, we define
\begin{equation}\label{def:kappa2}
\kappa_{2}(v):=\begin{cases} 2 & \mbox{if $v$ is the only vertex of $t$,}\\\smallskip
\rho_{2}((\lambda_{1},\ldots,\lambda_{j}),(\lambda_{j+1},\ldots,\lambda_{s})) & \parbox[t]{.45\textwidth}{if $v$ is the root vertex of $t$, $v$ is multi-branching, $\lambda_{1},\ldots,\lambda_{s}$ are the values of the direct descendants $v_{1},\ldots,v_{s}$ of $v$, respectively, and $j$ is as in $A2)$,}\\\smallskip
\kappa_{1}(v) & \mbox{if $v$ is not the root vertex of $t$}.
\end{cases}
\end{equation}
We also define
\begin{equation}\label{def:weighttree:2}
w_{2}(t):=\prod_{v}\kappa_{2}(v),
\end{equation}
the product taken over all vertices of $t$.

\begin{theorem}
For each integer $n\geq 1$,
\begin{equation}\label{eq:comb:alphaomega}
\omega_{n}=\sum_{\ov{n}\in\mathrm{C}(n)}\phi_{2}(\ov{n})\,\alpha(\ov{n}),
\end{equation}
where
\begin{equation}\label{def:func:phi2}
\phi_{2}(\ov{n}):=\sum_{t\in\mathcal{T}_{2}(\ov{n})}w_{2}(t).
\end{equation}
Moreover, for each $n\geq 1$,
\begin{equation}\label{eq:sum:phi2}
\sum_{\ov{n}\in\mathrm{C}(n)}\phi_{2}(\ov{n})=2n.
\end{equation}
\end{theorem}
\begin{proof}
According to \eqref{eq:formomega}, for any $n\geq 1$,
\begin{equation}\label{eq:omm}
\omega_{n}=\sum_{(\ov{p},\ov{q})\in\widehat{\mathrm{C}}(n)}\rho_{2}(\ov{p},\ov{q})\,m(\ov{p})\,m(\ov{q})
\end{equation}
Let $\ell(\ov{p})$ denote the number of components of $\ov{p}$, and $\ov{p}(j)$ the $j$-th component of $\ov{p}$. Then applying \eqref{formula:inv:1} we obtain
\[
m(\ov{p})=\prod_{j=1}^{\ell(\ov{p})} m_{\ov{p}(j)}=\sum_{(\ov{n}_{1},\ldots,\ov{n}_{\ell(\ov{p})})}\prod_{j=1}^{\ell(\ov{p})}\phi_{1}(\ov{n}_{j})\,\alpha(\ov{n}_{j})
\]
where the summation runs over all tuples $(\ov{n}_{1},\ldots,\ov{n}_{\ell(\ov{p})})\in\mathrm{C}(\ov{p}(1))\times\cdots\times \mathrm{C}(\ov{p}(\ell(\ov{p})))$. Inserting in the above identity the relations \eqref{def:funcphi}, we get
\begin{equation}\label{eq:mw1alpha}
m(\ov{p})=\sum_{(\ov{n}_{1},\ldots,\ov{n}_{\ell(\ov{p})})}\sum_{(t_{1},\ldots,t_{\ell(\ov{p})})}\prod_{j=1}^{\ell(\ov{p})}
w_{1}(t_{j}) \,\alpha(\ov{n}_{j}),
\end{equation}
where the inner summation is over all the tuples $(t_{1},\ldots,t_{\ell(\ov{p})})\in\mathcal{T}_{1}(\ov{n}_{1})\times\cdots\times\mathcal{T}_{1}(\ov{n}_{\ell(\ov{p})})$. Hence, \eqref{eq:mw1alpha} and the analogous expression for $m(\ov{q})$ applied to \eqref{eq:omm} give
\begin{equation}\label{eq:omegaanalysis}
\omega_{n}=\sum_{(\ov{p},\ov{q})\in\widehat{\mathrm{C}}(n)}
\sum_{\ov{\pi}}\sum_{\ov{\tau}}\sum_{\ov{\pi}'}\sum_{\ov{\tau}'}\rho_{2}(\ov{p},\ov{q})\left(\prod_{j=1}^{\ell(\ov{p})}
w_{1}(t_{j}) \,\alpha(\ov{n}_{j})\right)\left(\prod_{i=1}^{\ell(\ov{q})}
w_{1}(t_{i}') \,\alpha(\ov{n}_{i}')\right),
\end{equation}
where we used the abbreviations $\ov{\pi}=(\ov{n}_{1},\ldots,\ov{n}_{\ell(\ov{p})})$, $\ov{\tau}=(t_{1},\ldots,t_{\ell(\ov{p})})$, $\ov{\pi}'=(\ov{n}_{1}',\ldots,\ov{n}_{\ell(\ov{q})}')$, $\ov{\tau}'=(t_{1}',\ldots,t_{\ell(\ov{q})}')$.

In \eqref{eq:omegaanalysis} we shall distinguish different groups of terms. The first group is formed by only two terms; one is obtained by taking $(\ov{p},\ov{q})=((n),e)$, $\ov{n}_{1}=(n)$ and $t_{1}$ the tree whose only vertex has value $n$, and the other term is obtained by taking $(\ov{p},\ov{q})=(e,(n))$, $\ov{n}_{1}'=(n)$ and $t_{1}'$ as $t_{1}$ before. Both terms give the same contribution and we call this first group
\[
S_{1,n}:=2\alpha_{n}.
\]
The second group we distinguish is formed by all terms in \eqref{eq:omegaanalysis} obtained by taking $(\ov{p},\ov{q})=((n),e)$, $\ov{n}_{1}\in\mathrm{C}_{*}(n)$, cf. \eqref{def:Cstar}. The sum of all these terms is
\[
S_{2,n}:=\sum_{\ov{n}\in\mathrm{C}_{*}(n)}\sum_{t\in\mathcal{T}_{1}(\ov{n})}w_{1}(t)\,\alpha(\ov{n}).
\]
By symmetry, it is clear that we obtain the same expression if we take $(\ov{p},\ov{q})=(e,(n))$, $\ov{n}_{1}'\in\mathrm{C}_{*}(n)$. The third group that we distinguish is formed by all terms in \eqref{eq:omegaanalysis} obtained by taking $\ov{p}\in\mathrm{C}_{*}(n)$, $\ov{q}=e$, which we denote
\[
S_{3,n}:=\sum_{\ov{p}\in\mathrm{C}_{*}(n)}\sum_{(\ov{n}_{1},\ldots,\ov{n}_{\ell(\ov{p})})}\sum_{(t_{1},\ldots,t_{\ell(\ov{p})})}\rho_{1}(\ov{p})\prod_{j=1}^{\ell(\ov{p})}
w_{1}(t_{j}) \,\alpha(\ov{n}_{j}).
\]
We get the same expression if we take $\ov{p}=e$, $\ov{q}\in\mathrm{C}_{*}(n)$. Finally, the last group that we consider is formed by all the terms in \eqref{eq:omegaanalysis} obtained by taking
\[
(\ov{p},\ov{q})\in\bigcup_{j=1}^{n-1} \mathrm{C}(j)\times\mathrm{C}(n-j),
\]
i.e., $(\ov{p},\ov{q})\in\widehat{\mathrm{C}}(n)$ with $\ov{p}\neq e$, $\ov{q}\neq e$. We call $S_{4,n}$ the sum of all the terms in this last group.

Hence, we have
\[
\omega_{n}=S_{1,n}+2 S_{2,n}+2 S_{3,n}+S_{4,n}.
\]
But $S_{3,n}=-S_{2,n}$, since in virtue of \eqref{eq:mn:summation}, \eqref{formula:inv:1} and \eqref{def:funcphi},
\[
S_{3,n}=\alpha_{n}-m_{n}=\alpha_{n}-\sum_{\ov{n}\in\mathrm{C}(n)}\sum_{t\in\mathcal{T}_{1}(\ov{n})} w_{1}(t)\,\alpha(\ov{n})=-S_{2,n}.
\]
We conclude that
\begin{equation}\label{eq:reductionomega}
\omega_{n}=2\alpha_{n}+S_{4,n}.
\end{equation}

To finish the proof of \eqref{eq:comb:alphaomega}, we will show that there is a bijective correspondence between terms in $2\alpha_{n}+S_{4,n}$ and trees in the classes $\mathcal{T}_{2}(\ov{n})$, $\ov{n}\in\mathrm{C}(n)$. This correspondence can be constructed in a way similar to the construction of the map $T_{1}$ in the proof of \eqref{formula:inv:1}. We describe this construction below.

First, observe that
\begin{equation}\label{eq:alphaphi2}
2\alpha_{n}=w_{2}(t_{0})\,\alpha((n))=\phi_{2}((n))\,\alpha((n)),
\end{equation}
where $t_{0}$ represents the tree with only one vertex with value $n$. Now we focus on the terms that appear in $S_{4,n}$. These terms are parametrized by the elements in the set $\mathcal{S}_{2}$ consisting of all the tuples $((\ov{p},\ov{q}), \ov{\pi},\ov{\tau}, \ov{\pi}',\ov{\tau}')$ satisfying
\[
\begin{aligned}
(\ov{p},\ov{q}) & \in\bigcup_{j=1}^{n-1} \mathrm{C}(j)\times\mathrm{C}(n-j)\\
\ov{\pi} & = (\ov{n}_{1},\ldots,\ov{n}_{\ell(\ov{p})})\in\mathrm{C}(\ov{p}(1))\times\cdots\times \mathrm{C}(\ov{p}(\ell(\ov{p})))\\
\ov{\tau} & = (t_{1},\ldots,t_{\ell(\ov{p})})\in\mathcal{T}_{1}(\ov{n}_{1})\times\cdots\times\mathcal{T}_{1}(\ov{n}_{\ell(\ov{p})})\\
\ov{\pi}' & = (\ov{n}_{1}',\ldots,\ov{n}_{\ell(\ov{q})}')\in\mathrm{C}(\ov{q}(1))\times\cdots\times \mathrm{C}(\ov{q}(\ell(\ov{q})))\\
\ov{\tau}' & =(t_{1}',\ldots,t_{\ell(\ov{q})}')\in\mathcal{T}_{1}(\ov{n}_{1}')\times\cdots\times\mathcal{T}_{1}(\ov{n}_{\ell(\ov{q})}')
\end{aligned}
\]
see \eqref{eq:omegaanalysis}. Let $\mathcal{T}_{2}$ be the collection of all trees belonging to the classes $\mathcal{T}_{2}(\ov{n})$, $\ov{n}\in\mathrm{C}_{*}(n)$. Let $T_{2}:\mathcal{S}_{2}\longrightarrow\mathcal{T}_{2}$ be the map that assigns to each element $((\ov{p},\ov{q}), \ov{\pi},\ov{\tau}, \ov{\pi}',\ov{\tau}')$ the tree $t$ constructed using the following multi-step procedure:
\begin{itemize}
\item[1)] Construct the admissible tree with two levels (levels $0$ and $1$), where level $1$ is formed, from left to right, by the vertices with values $\ov{p}(1),\ldots,\ov{p}(\ell(\ov{p})), \ov{q}(1),\ldots,\ov{q}(\ell(\ov{q}))$. Color blue all the edges connecting the root with the vertices with values $\ov{p}(1),\ldots,\ov{p}(\ell(\ov{p}))$, and color red the remaining edges.

\item[2)]  Let $d_{j}$ denote the height of the tree $t_{j}$, $1\leq j\leq \ell(\ov{p})$, and $d_{l}'$ the height of the tree $t_{l}'$, $1\leq l\leq \ell(\ov{q})$. If $d$ is the maximum of all the values $d_{j}$ and $d_{l}'$, construct new trees $\widehat{t}_{j}$ (and $\widehat{t}_{l}'$) by performing extensions of the trees $t_{j}$ (respectively, $t_{l}'$) with $d-d_{j}$ units (respectively, $d-d_{l}'$ units). Color blue all the edges of the trees
$\widehat{t}_{j}$, $1\leq j\leq \ell(\ov{p})$, and color red all the edges of the trees $\widehat{t}_{l}'$, $1\leq l\leq \ell(\ov{q})$.

\item[3)] For each $j=1,\ldots,\ell(\ov{p}),$ append the tree $\widehat{t}_{j}$ to the tree constructed in step $1)$ by using the vertex with value $\ov{p}(j)$ as the root vertex of the tree $\widehat{t}_{j}$, and do the same with the trees $\widehat{t}_{l}'$ and the vertices with values $\ov{q}(l)$. Let $t$ be the tree obtained after completing this step.
\end{itemize}

 Since $\ov{p}\neq e$, $\ov{q}\neq e$, the tree $t$ is indeed a tree in the class $\mathcal{T}_{2}(\ov{n})$, for some $\ov{n}\in\mathrm{C}_{*}(n)$. The entries of $\ov{n}$ are obtained by concatenation of the entries of $\ov{n}_{1},\ldots,\ov{n}_{\ell(\ov{p})}, \ov{n}_{1}',\ldots,\ov{n}_{\ell(\ov{q})}',$ in this order from left to right. It is also clear from the construction of $t$ and the definitions \eqref{def:kappa2}--\eqref{def:weighttree:2} that
\begin{equation}\label{eq:finalrel}
\rho_{2}(\ov{p},\ov{q})\left(\prod_{j=1}^{\ell(\ov{p})}
w_{1}(t_{j}) \,\alpha(\ov{n}_{j})\right)\left(\prod_{i=1}^{\ell(\ov{q})}
w_{1}(t_{i}') \,\alpha(\ov{n}_{i}')\right)=w_{2}(t)\,\alpha(\ov{n}).
\end{equation}
Moreover, it is easy to see that the map $T_{2}$ is a bijection. Hence, from \eqref{eq:reductionomega}, \eqref{eq:alphaphi2} and \eqref{eq:finalrel} we deduce that
\[
\omega_{n}=2\alpha_{n}+\sum_{\ov{n}\in\mathrm{C}_{*}(n)}\sum_{t\in\mathcal{T}_{2}(\ov{n})}w_{2}(t)\,\alpha(\ov{n}),
\]
which is the identity \eqref{eq:comb:alphaomega}.

Now we justify \eqref{eq:sum:phi2}. We proceed as in the proof of \eqref{eq:sum:inv:1}, and we formally set $\alpha_{n}=1$ for all $n\geq 0$. The goal is to show that in this case $\omega_{n}=\sum_{\ov{n}\in\mathrm{C}(n)}\phi_{2}(\ov{n})=2n$. According to \eqref{eq:partalpham}, in this situation we have $m_{1}=1$ and $m_{k}=0$ for all $k\geq 2$, hence in virtue of \eqref{eq:omm} we get
\[
\omega_{n}=\sum_{(\ov{p},\ov{q})\in\widehat{\mathrm{C}}(n)}\rho_{2}(\ov{p},\ov{q})\,m(\ov{p})\,m(\ov{q})
=\sum_{j=0}^{n}\rho_{2}(\mathbf{1}_{j}, \mathbf{1}_{n-j})=2n,
\]
where for $j>0$, $\mathbf{1}_{j}$ denotes the vector $(1,\ldots,1)\in\mathrm{C}(j)$ with all its $j$ entries equal to $1$, and $\mathbf{1}_{0}=e\in\mathrm{C}(0)$.
\end{proof}

In this section we also analyze the relations that express the quantities $m_{n}$ and $\alpha_{n}$ in terms of the quantities $\omega_{k}$, $k=0,\ldots,n$. The first few relations between these sequences take the form
\begin{align*}
m_{1} & =\frac{1}{2} \omega_{1}\\
m_{2} & =\frac{1}{2} \omega_{2} - \frac{1}{2} \omega_{1}^2\\
m_{3} & =\frac{1}{2} \omega_{3} - \frac{3}{2} \omega_2 \omega_1 + \frac{9}{8} \omega_1^3\\
m_{4} & =\frac{1}{2} \omega_4 - 2 \omega_3 \omega_1 + 7 \omega_2 \omega_1^2 - \frac{3}{2} \omega_2^2 - \frac{17}{4} \omega_{1}^4\\
m_{5} & =\frac{1}{2} \omega_{5}-\frac{5}{2} \omega_4\omega_1-5 \omega_3\omega_2+\frac{95}{8} \omega_3\omega_1^2+\frac{145}{8} \omega_2^2 \omega_1-45 \omega_2\omega_1^3+\frac{365}{16} \omega_1^5\\
\alpha_{1} & =\frac{1}{2}\omega_{1}\\
\alpha_{2} & =\frac{1}{2} \omega_{2}-\frac{1}{4}\omega_1^2\\
\alpha_{3} & =\frac{1}{2} \omega_{3}-\frac{3}{4}\omega_2\omega_1+\frac{1}{2}\omega_{1}^3\\
\alpha_{4} & =\frac{1}{2} \omega_{4}-\omega_{3}\omega_1-\frac{3}{4}\omega_2^2+\frac{25}{8} \omega_2\omega_1^2-\frac{29}{16} \omega_{1}^4\\
\alpha_{5} & =\frac{1}{2} \omega_{5}-\frac{5}{4} \omega_4 \omega_1-\frac{5}{2} \omega_3\omega_2+\frac{21}{4} \omega_3 \omega_1^2+\frac{33}{4} \omega_2^2\omega_1-\frac{309}{16}\omega_2\omega_1^3+\frac{19}{2}\omega_1^5.
\end{align*}

We need to introduce two more classes of bi-colored trees. The first one is defined as follows. For $n\geq 1$ and $\ov{n}=(n_{0},\ldots,n_{r})\in\mathrm{C}(n)$, we say that a tree $t$ belongs to the class $\mathcal{T}_{3}(\ov{n})$, if in addition to $T1)$--$T4)$, $t$ satisfies the following conditions:

\begin{itemize}
\item[$B1)$] Every edge of the tree is either a \emph{blue edge} ($b$-edge) or a \emph{red edge} ($r$-edge).

\item[$B2)$] If a vertex $v$ is multi-branching and $v_{1}, v_{2},\ldots,v_{s},$ $s\geq 2$, are the direct descendants of $v$, listed from left to right, then there exists $0\leq j\leq s$ such that the edges connecting $v$ with $v_{1},\ldots,v_{j}$ are all $b$-edges, and the edges connecting $v$ with $v_{j+1},\ldots,v_{s}$ are all $r$-edges. If $j=0$, then we understand that all the edges are $r$-edges, and if $j=s$, then all the edges are $b$-edges.

\item[$B3)$] If a vertex $v$ has only one direct descendant, then the edge connecting $v$ with its parent and the edge connecting $v$ with its direct descendant have the same color.
\end{itemize}

Note that these rules uniquely determine the color of each edge. For any vertex $v$ of a tree $t\in\mathcal{T}_{3}(\ov{n})$, we define
\begin{equation}\label{def:kappa3}
\kappa_{3}(v):=\begin{cases}
\frac{1}{2} & \parbox[t]{.4\textwidth}{if $v$ is the only vertex of the tree, or if $v$ is not multi-branching and is a direct descendant of a multi-branching vertex,}\\[0.9cm]
-\frac{1}{2}\,\rho_{2}((\lambda_{1},\ldots,\lambda_{j}),(\lambda_{j+1},\ldots,\lambda_{s})) &
\parbox[t]{.4\textwidth}{if $v$ is multi-branching, $\lambda_{1},\ldots,\lambda_{s}$, $s\geq 2$, are the values of the direct descendants $v_{1},\ldots,v_{s}$ of $v$, respectively, and $j$ is as in $B2)$,}\\[1.2cm]
1 & \mbox{otherwise}.
\end{cases}
\end{equation}
In this definition, we understand $(\lambda_{1},\ldots,\lambda_{j})=e$ if $j=0$, and $(\lambda_{j+1},\ldots,\lambda_{s})=e$ if $j=s$. The third case in \eqref{def:kappa3} refers to a vertex that is not multi-branching and is not a direct descendant of a multi-branching vertex.

Finally, for a tree $t\in\mathcal{T}_{3}(\ov{n})$ we define
\begin{equation}\label{def:weighttree:3}
w_{3}(t):=\prod_{v}\kappa_{3}(v),
\end{equation}
where the product is taken over all vertices of $t$.

\begin{theorem}\label{theo:invrelmomega}
For each integer $n\geq 1$,
\begin{equation}\label{formula:inv:2}
m_{n}=\sum_{\ov{n}\in\mathrm{C}(n)}\phi_{3}(\ov{n})\,\omega(\ov{n}),
\end{equation}
where
\begin{equation}\label{def:funcphi3}
\phi_{3}(\ov{n}):=\sum_{t\in\mathcal{T}_{3}(\ov{n})}w_{3}(t).
\end{equation}
\end{theorem}
\begin{proof}
The argument used here is the same used in the proof of \eqref{formula:inv:1}, but for the sake of clarity in the exposition we reproduce it.

First, the result is trivially true for $n=1$, since $\phi_{3}((1))=1/2$ and $m_{1}=\omega_{1}/2$. Assume that \eqref{formula:inv:2} holds for all values $n=1,\ldots,k-1,$ $k\geq 2,$ and let us prove that it also holds for $n=k$.

It follows from \eqref{eq:formomega} that
\begin{equation}\label{eq:mk:int}
m_{k}=\frac{\omega_{k}}{2}-\frac{1}{2}\sum_{(\ov{p},\ov{q})\in\widehat{\mathrm{C}}_{*}(k)}\rho_{2}(\ov{p},\ov{q})\, m(\ov{p})\, m(\ov{q}),
\end{equation}
where
\[
\widehat{\mathrm{C}}_{*}(k):=\widehat{\mathrm{C}}(k)\setminus\{((k),e), (e,(k))\}.
\]
Given a pair $(\ov{p},\ov{q})\in\widehat{\mathrm{C}}_{*}(k)$, if we apply the induction hypothesis to each factor in $m(\ov{p})$ and $m(\ov{q})$, we obtain
\begin{align}
m(\ov{p}) & =\prod_{j=1}^{\ell(\ov{p})} m_{\ov{p}(j)}=\sum_{(\ov{n}_{1},\ldots,\ov{n}_{\ell(\ov{p})})} \prod_{j=1}^{\ell(\ov{p})}\phi_{3}(\ov{n}_{j})\,\omega(\ov{n}_{j})=\sum_{(\ov{n}_{1},\ldots,\ov{n}_{\ell(\ov{p})})}\sum_{(t_{1},\ldots,t_{\ell(\ov{p})})}\prod_{j=1}^{\ell(\ov{p})}w_{3}(t_{j})\,\omega(\ov{n}_{j}),\label{eq:mp}\\
m(\ov{q}) & =\prod_{i=1}^{\ell(\ov{q})} m_{\ov{q}(i)}=\sum_{(\ov{n}_{1}',\ldots,\ov{n}_{\ell(\ov{q})}')}\prod_{i=1}^{\ell(\ov{q})}\phi_{3}(\ov{n}_{i}')\,\omega(\ov{n}_{i}')=\sum_{(\ov{n}_{1}',\ldots,\ov{n}_{\ell(\ov{q})}')}\sum_{(t_{1}',\ldots,t_{\ell(\ov{q})}')}\prod_{i=1}^{\ell(\ov{q})}w_{3}(t_{i}')\,\omega(\ov{n}_{i}'),\label{eq:mq}
\end{align}
where the summations in \eqref{eq:mp} are taken over $(\ov{n}_{1},\ldots,\ov{n}_{\ell(\ov{p})})\in\mathrm{C}(\ov{p}(1))\times\cdots\times \mathrm{C}(\ov{p}(\ell(\ov{p})))$, $(t_{1},\ldots,t_{\ell(\ov{p})})\in\mathcal{T}_{3}(\ov{n}_{1})\times \cdots \times\mathcal{T}_{3}(\ov{n}_{\ell(\ov{p})})$, and the summations in \eqref{eq:mq} are taken over $(\ov{n}_{1}',\ldots,\ov{n}_{\ell(\ov{q})}')\in\mathrm{C}(\ov{q}(1))\times\cdots\times \mathrm{C}(\ov{q}(\ell(\ov{q})))$, $(t_{1}',\ldots,t_{\ell(\ov{q})}')\in\mathcal{T}_{3}(\ov{n}_{1}')\times \cdots \times \mathcal{T}_{3}(\ov{n}_{\ell(\ov{q})}')$.

From \eqref{eq:mk:int} and \eqref{eq:mp}--\eqref{eq:mq} we obtain
\begin{equation}\label{eq:mkexp}
m_{k}=\frac{\omega_{k}}{2}-\frac{1}{2}\sum_{(\ov{p},\ov{q})\in\widehat{\mathrm{C}}_{*}(k)}\sum_{\ov{\pi}}\sum_{\ov{\tau}}\sum_{\ov{\pi}'}\sum_{\ov{\tau}'}\rho_{2}(\ov{p},\ov{q})\left(\prod_{j=1}^{\ell(\ov{p})}w_{3}(t_{j})\,\omega(\ov{n}_{j})\right)\left(\prod_{i=1}^{\ell(\ov{q})}w_{3}(t_{i}')\,\omega(\ov{n}_{i}')\right)
\end{equation}
where we have the same abbreviations used in \eqref{eq:omegaanalysis}.

One can show, as in the proof of \eqref{eq:finalrel}, that for each term in the summation in \eqref{eq:mkexp} there exists a unique tree $t\in\mathcal{T}_{3}(\ov{n})$ associated with a vector $\ov{n}\in\mathrm{C}_{*}(k)$ such that
\begin{equation}\label{keyrel}
-\frac{1}{2}\,\rho_{2}(\ov{p},\ov{q})\left(\prod_{j=1}^{\ell(\ov{p})}w_{3}(t_{j})\,\omega(\ov{n}_{j})\right)\left(\prod_{i=1}^{\ell(\ov{q})}w_{3}(t_{i}')\,\omega(\ov{n}_{i}')\right)=w_{3}(t)\,\omega(\ov{n}).
\end{equation}
To make the argument explicit, let us consider the set $\mathcal{S}_{3}$ consisting of all the tuples $((\ov{p},\ov{q}),\ov{\pi},\ov{\tau},\ov{\pi}',\ov{\tau}')$ where
\begin{equation}\label{descStilde}
\begin{aligned}
(\ov{p},\ov{q}) & \in \widehat{\mathrm{C}}_{*}(k)\\
\ov{\pi} & = (\ov{n}_{1},\ldots,\ov{n}_{\ell(\ov{p})})\in\mathrm{C}(\ov{p}(1))\times\cdots\times \mathrm{C}(\ov{p}(\ell(\ov{p})))\\
\ov{\tau} & = (t_{1},\ldots,t_{\ell(\ov{p})})\in\mathcal{T}_{3}(\ov{n}_{1})\times\cdots\times\mathcal{T}_{3}(\ov{n}_{\ell(\ov{p})})\\
\ov{\pi}' & = (\ov{n}_{1}',\ldots,\ov{n}_{\ell(\ov{q})}')\in\mathrm{C}(\ov{q}(1))\times\cdots\times \mathrm{C}(\ov{q}(\ell(\ov{q})))\\
\ov{\tau}' & =(t_{1}',\ldots,t_{\ell(\ov{q})}')\in\mathcal{T}_{3}(\ov{n}_{1}')\times\cdots\times\mathcal{T}_{3}(\ov{n}_{\ell(\ov{q})}')
\end{aligned}
\end{equation}
and let $\mathcal{T}_{3}$ be the collection of all trees in the classes $\mathcal{T}_{3}(\ov{n})$, $\ov{n}\in\mathrm{C}_{*}(k)$. If $\ov{p}=e$ (or $\ov{q}=e$) in \eqref{descStilde}, then we understand that the corresponding elements in $\mathcal{S}_{3}$ are of the form $((e,\ov{q}),\ov{\pi}',\ov{\tau}')$ (respectively $((\ov{p},e),\ov{\pi},\ov{\tau})$). Now consider the map $T_{3}:\mathcal{S}_{3}\longrightarrow\mathcal{T}_{3}$ that assigns to each element $((\ov{p},\ov{q}),\ov{\pi},\ov{\tau},\ov{\pi}',\ov{\tau}')$ the tree $t$ constructed using the following procedure:
\begin{itemize}
\item[1)] Construct the admissible tree with two levels (levels $0$ and $1$), where level $1$ is formed, from left to right, by the vertices with values $\ov{p}(1),\ldots,\ov{p}(\ell(\ov{p})), \ov{q}(1),\ldots,\ov{q}(\ell(\ov{q}))$. Color blue the edges connecting the root with the vertices with values $\ov{p}(1),\ldots,\ov{p}(\ell(\ov{p}))$, and color red the rest. If $\ov{p}=e$ (or $\ov{q}=e$), then all edges in this tree are red (respectively blue). Note that this tree contains at least two edges.

\item[2)]  Let $d_{j}$ denote the height of the tree $t_{j}$, $1\leq j\leq \ell(\ov{p})$, and $d_{i}'$ the height of the tree $t_{i}'$, $1\leq i\leq \ell(\ov{q})$. If $d$ is the maximum of all the values $d_{j}$ and $d_{i}'$, construct new trees $\widehat{t}_{j}$ (and $\widehat{t}'_{i}$) by performing an extension of the trees $t_{j}$ (respectively $t_{i}'$) with $d-d_{j}$ units (respectively $d-d_{i}'$ units). The coloring of the new edges that are added in the extensions follows the following rules. Assume that $d-d_{j}>0$ (otherwise $t_{j}=\widehat{t}_{j}$ and there are no additional edges). If the tree $t_{j}$ has height $d_{j}>0$, then each new edge added to form $\widehat{t}_{j}$ is colored so that the rule $B3)$ holds for all vertices in $\widehat{t}_{j}$ that are not multi-branching. If $t_{j}$ has height $d_{j}=0$, then each new edge added to $t_{j}$ is colored blue. The same rules are valid for the trees $\widehat{t}_{i}'$, except that red is used instead of blue in the case that $t_{i}'$ has height zero.

\item[3)] For each $j=1,\ldots,\ell(\ov{p}),$ append the tree $\widehat{t}_{j}$ to the tree constructed in step $1)$ by using the vertex with value $\ov{p}(j)$ as the root vertex of the tree $\widehat{t}_{j}$, and do the same with the trees $\widehat{t}_{i}'$ and the vertices with values $\ov{q}(i)$. Let $t$ be the tree obtained after completing this step.
\end{itemize}

From this construction we easily deduce that $t$ is a tree in the class $\mathcal{T}_{3}(\ov{n})$ for some $\ov{n}\in\mathrm{C}_{*}(k)$, and it is clear that \eqref{keyrel} holds. Moreover, the map $T_{3}:\mathcal{S}_{3}\longrightarrow\mathcal{T}_{3}$ is a bijection and therefore we can write
\[
m_{k}=\frac{\omega_{k}}{2}+\sum_{\ov{n}\in\mathrm{C}_{*}(k)}\sum_{t\in\mathcal{T}_{3}(\ov{n})}w_{3}(t)\,\omega(\ov{n}),
\]
which is the desired identity.
\end{proof}

For our last result in this section, we introduce the following class of trees. Given $n\geq 1$ and $\ov{n}=(n_{0},\ldots,n_{r})\in\mathrm{C}(n)$, we say that a tree $t$ belongs to the class $\mathcal{T}_{4}(\ov{n})$, if in addition to $T1)$--$T4)$, $t$ satisfies the following conditions:

\begin{itemize}
\item[$C1)$] Every edge of the tree is either a \emph{blue edge} ($b$-edge) or a \emph{red edge} ($r$-edge).

\item[$C2)$] If $t$ does not reduce to a single vertex and $v$ is the root vertex of $t$, then condition $A2)$ holds for this vertex. That is, if $v_{1},\ldots,v_{s},$ $s\geq 2$, are the direct descendants of $v$, listed from left to right, then there exists $1\leq j\leq s-1$ such that the edges connecting $v$ with $v_{1},\ldots,v_{j}$ are all $b$-edges, and the edges connecting $v$ with $v_{j+1},\ldots,v_{s}$ are all $r$-edges.

\item[$C3)$] If $v$ is a vertex in $t$ different from the root vertex, then condition $B2)$ holds for $v$ if $v$ is multi-branching, and condition $B3)$ holds for $v$ if $v$ is not multi-branching.
\end{itemize}

These rules uniquely determine the color of each edge. Also note that among the edges connecting the root vertex with its direct descendants, there are at least one $b$-edge and at least one $r$-edge. For any vertex $v$ of a tree $t\in\mathcal{T}_{4}(\ov{n})$, let
\[
\kappa_{4}(v):=\begin{cases}
\frac{1}{2} & \parbox[t]{.55\textwidth}{if $v$ is the only vertex of $t$,}\\ \smallskip
-\frac{1}{2}\,\rho_{2}((\lambda_{1},\ldots,\lambda_{j}),(\lambda_{j+1},\ldots,\lambda_{s})) &
\parbox[t]{.4\textwidth}{if $v$ is the root vertex of $t$, $v$ is multi-branching, $\lambda_{1},\ldots,\lambda_{s}$ are the values of the direct descendants $v_{1},\ldots,v_{s}$ of $v$, respectively, and $j$ is as in $C2)$,}\\ \smallskip
\kappa_{3}(v) & \mbox{if $v$ is not the root vertex of $t$}.
\end{cases}
\]

We define
\[
w_{4}(t):=\prod_{v}\kappa_{4}(v),
\]
the product taken over all vertices of $t$.

\begin{theorem}\label{theo:combalphaomega}
For each $n\geq 1$,
\begin{equation}\label{eq:alphafomega}
\alpha_{n}=\sum_{\ov{n}\in\mathrm{C}(n)}\phi_{4}(\ov{n})\,\omega(\ov{n}),
\end{equation}
where
\begin{equation}\label{def:phi4}
\phi_{4}(\ov{n}):=\sum_{t\in\mathcal{T}_{4}(\ov{n})}w_{4}(t).
\end{equation}
\end{theorem}
\begin{proof}
According to \eqref{eq:formalpha}, we have
\begin{equation}\label{eq:auxalpham}
\alpha_{n}=\sum_{\ov{p}\in\mathrm{C}(n)}\rho_{1}(\ov{p})\,m(\ov{p})=m_{n}+\beta_{n},
\end{equation}
where
\[
\beta_{n}:=\sum_{\ov{p}\in\mathrm{C}_{*}(n)} \rho_{1}(\ov{p})\,m(\ov{p}).
\]
Using \eqref{eq:mp}, we can rewrite $\beta_{n}$ as
\[
\beta_{n}=\sum_{\ov{p}\in\mathrm{C}_{*}(n)}\sum_{(\ov{n}_{1},\ldots,\ov{n}_{\ell(\ov{p})})}\sum_{(t_{1},\ldots,t_{\ell(\ov{p})})}\rho_{1}(\ov{p})\left(\prod_{j=1}^{\ell(\ov{p})}w_{3}(t_{j})\,\omega(\ov{n}_{j})\right),
\]
and, reproducing \eqref{eq:mkexp}, we have
\begin{equation}\label{eq:mkexp:again}
m_{n}=\frac{\omega_{n}}{2}-\frac{1}{2}\sum_{(\ov{p},\ov{q})\in\widehat{\mathrm{C}}_{*}(n)}\sum_{\ov{\pi}}\sum_{\ov{\tau}}\sum_{\ov{\pi}'}\sum_{\ov{\tau}'}\rho_{2}(\ov{p},\ov{q})\left(\prod_{j=1}^{\ell(\ov{p})}w_{3}(t_{j})\,\omega(\ov{n}_{j})\right)\left(\prod_{i=1}^{\ell(\ov{q})}w_{3}(t_{i}')\,\omega(\ov{n}_{i}')\right)
\end{equation}
where the summation indexes have the same meaning as in the referenced formulas.

If we distinguish in the summation in \eqref{eq:mkexp:again} the terms with $\ov{p}=e$, we readily see that the sum of all these terms is exactly $\beta_{n}$ (here we are disregarding the prefactor $-\frac{1}{2}$). Likewise, the sum of all the terms obtained by taking $\ov{q}=e$ is $\beta_{n}$. The remaining terms in the summation are those corresponding to the choice $(\ov{p},\ov{q})\in\widehat{\mathrm{C}}_{*}(n)$, $\ov{p}\neq e$, $\ov{q}\neq e$, or equivalently, $(\ov{p},\ov{q})\in \widetilde{\mathrm{C}}(n):=\bigcup_{j=1}^{n-1}\mathrm{C}(j)\times\mathrm{C}(n-j)$. Therefore, these considerations together with \eqref{eq:auxalpham} and \eqref{eq:mkexp:again} imply that
\[
\alpha_{n}=\frac{\omega_{n}}{2}-\frac{1}{2}\sum_{(\ov{p},\ov{q})\in \widetilde{\mathrm{C}}(n)}\sum_{\ov{\pi}}\sum_{\ov{\tau}}\sum_{\ov{\pi}'}\sum_{\ov{\tau}'}\rho_{2}(\ov{p},\ov{q})\left(\prod_{j=1}^{\ell(\ov{p})}w_{3}(t_{j})\,\omega(\ov{n}_{j})\right)\left(\prod_{i=1}^{\ell(\ov{q})}w_{3}(t_{i}')\,\omega(\ov{n}_{i}')\right).
\]

The rest of the argument can be completed as in the proof of \eqref{keyrel}. Indeed, we can consider the set $\mathcal{S}_{4}$ consisting of all the tuples $((\ov{p},\ov{q}),\ov{\pi},\ov{\tau},\ov{\pi}',\ov{\tau}')$ satisfying \eqref{descStilde} with $\widehat{\mathrm{C}}_{*}(k)$ replaced by $\widetilde{\mathrm{C}}(n)=\bigcup_{j=1}^{n-1}\mathrm{C}(j)\times\mathrm{C}(n-j)$, and consider the set $\mathcal{T}_{4}$ consisting of all trees in the classes $\mathcal{T}_{4}(\ov{n})$, $\ov{n}\in\mathrm{C}_{*}(\ov{n})$. If we define the map $T_{4}:\mathcal{S}_{4}\longrightarrow\mathcal{T}_{4}$ the same way $T_{3}$ was defined in the proof of \eqref{keyrel}, then $T_{4}$ is a bijection, which will imply that
\[
\alpha_{n}=\frac{\omega_{n}}{2}+\sum_{\ov{n}\in\mathrm{C}_{*}(n)}\sum_{t\in\mathcal{T}_{4}(\ov{n})}w_{4}(t)\,\omega(\ov{n})
\]
and this is \eqref{eq:alphafomega}.
\end{proof}

\begin{figure}
\begin{center}
\begin{tikzpicture}[level distance=10mm,
every node/.style={inner sep=1pt,circle,draw},
level 1/.style={sibling distance=10mm}]
\node at (0,0) {$2$};
\node at (2,0) {$2$}
child {node{$1$} edge from parent node[left, draw=none, scale=0.8] {b}}
child {node{$1$} edge from parent node[right, draw=none, scale=0.8] {b}};
\node at (4,0) {$2$}
child {node{$1$} edge from parent node[left, draw=none, scale=0.8] {b}}
child {node{$1$}};
\node[draw=none, scale=0.8] at (4.36,-0.52) {r};
\node at (6,0) {$2$}
child {node{$1$} edge from parent node[left, draw=none, scale=0.8] {r}}
child {node{$1$}};
\node[draw=none, scale=0.8] at (6.36,-0.52) {r};
\node[draw=none] at (0,-0.5) {$\widetilde{\mathbf{t}}_{1}$};
\node[draw=none] at (2,-1.5) {$\widetilde{\mathbf{t}}_{2}$};
\node[draw=none] at (4,-1.5) {$\widetilde{\mathbf{t}}_{3}$};
\node[draw=none] at (6,-1.5) {$\widetilde{\mathbf{t}}_{4}$};
\end{tikzpicture}
\end{center}
\caption{The four trees that form the set $\bigcup_{\ov{n}\in\mathrm{C}(2)}\mathcal{T}_{3}(\ov{n}).$}
\label{exampletrees1}
\end{figure}
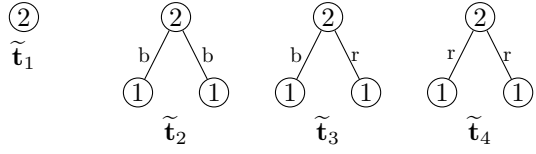

We finish this section illustrating the formulas \eqref{formula:inv:2} and \eqref{eq:alphafomega} in the cases $n=2, 3$. 

The four trees that form the set $\bigcup_{\ov{n}\in\mathrm{C}(2)}\mathcal{T}_{3}(\ov{n})$ are shown in Fig.~\ref{exampletrees1}. We have $w_{3}(\widetilde{\mathbf{t}}_{1})=1/2$, $w_{3}(\widetilde{\mathbf{t}}_{2})=w_{3}(\widetilde{\mathbf{t}}_{4})=-1/8$, and $w_{3}(\widetilde{\mathbf{t}}_{3})=-1/4$, hence according to \eqref{formula:inv:2},
\[
m_{2}=w_{3}(\widetilde{\mathbf{t}}_{1})\,\omega_{2}+(w_{3}(\widetilde{\mathbf{t}}_{2})+w_{3}(\widetilde{\mathbf{t}}_{3})+w_{3}(\widetilde{\mathbf{t}}_{4}))\,\omega_{1}^2=\frac{1}{2}\,\omega_{2}-\frac{1}{2}\,\omega_{1}^2.
\]
The two trees that form the set $\bigcup_{\ov{n}\in\mathrm{C}(2)}\mathcal{T}_{4}(\ov{n})$ are the trees $\widetilde{\mathbf{t}}_{1}$ and $\widetilde{\mathbf{t}}_{3}$ shown in Fig.~\ref{exampletrees1}. Since $w_{4}(\widetilde{\mathbf{t}}_{1})=1/2$, $w_{4}(\widetilde{\mathbf{t}}_{3})=-1/4$, formula \eqref{eq:alphafomega} gives
\[
\alpha_{2}=w_{4}(\widetilde{\mathbf{t}}_{1})\,\omega_{2}+w_{4}(\widetilde{\mathbf{t}}_{3})\,\omega_{1}^2=\frac{1}{2}\,\omega_{2}-\frac{1}{4}\,\omega_{1}^2.
\]

Now we check the formulas in the case $n=3$. The set $\bigcup_{\ov{n}\in\mathrm{C}(3)}\mathcal{T}_{3}(\ov{n})$ consists of twenty-nine trees that are shown in Fig.~\ref{exampletrees2}. The reader can check that the values of the $w_{3}$-weights of these trees are
\begin{align*}
w_{3}(\mathbf{t}_{1}) & =\frac{1}{2}\\
w_{3}(\mathbf{t}_{j}) & =-\frac{1}{4},\quad j\in\{2, 4\}\\
w_{3}(\mathbf{t}_{j}) & =-\frac{3}{8},\quad j\in\{3, 6\}\\
w_{3}(\mathbf{t}_{j}) & =-\frac{1}{8}, \quad j\in\{5, 7, 9, 10\}\\
w_{3}(\mathbf{t}_{j}) & =-\frac{1}{16}, \quad j\in\{8, 11\}\\
w_{3}(\mathbf{t}_{j}) & =\frac{1}{16}, \quad j\in\{12, 14, 18, 20, 22, 28\}\\
w_{3}(\mathbf{t}_{j}) & =\frac{1}{8}, \quad j\in\{13, 19\}\\
w_{3}(\mathbf{t}_{j}) & =\frac{3}{32}, \quad j\in\{15, 17, 24, 26\}\\
w_{3}(\mathbf{t}_{j}) & =\frac{3}{16}, \quad j\in\{16, 25\}\\
w_{3}(\mathbf{t}_{j}) & =\frac{1}{32}, \quad j\in\{21, 23, 27, 29\}
\end{align*} 
which gives $\phi_{3}((3))=1/2$, $\phi_{3}((2,1))+\phi_{3}((1,2))=\sum_{j=2}^{7}w_{3}(\mathbf{t}_{j})=-3/2$, and $\phi_{3}((1,1,1))=\sum_{j=8}^{29}w_{3}(\mathbf{t}_{j})=9/8$. This shows that
\[
m_{3}=\sum_{\ov{n}\in\mathrm{C}(3)}\phi_{3}(\ov{n})\,\omega(\ov{n})=\frac{1}{2}\,\omega_{3}-\frac{3}{2}\,\omega_{2}\,\omega_{1}+\frac{9}{8}\,\omega_{1}^{3}.
\]
There are eleven trees in the collection 
\[
\bigcup_{\ov{n}\in\mathrm{C}(3)}\mathcal{T}_{4}(\ov{n})=\{\mathbf{t}_{j}: j\in J\},\quad J=\{1, 3, 6, 9, 10, 15, 16, 17, 24, 26\}, 
\]
where the right-hand side refers to trees in Fig.~\ref{exampletrees2}. The reader can check that
\begin{align*}
w_{4}(\mathbf{t}_{1}) & =\frac{1}{2}\\
w_{4}(\mathbf{t}_{j}) & =-\frac{3}{8},\quad j\in\{3, 6\}\\
w_{4}(\mathbf{t}_{j}) & =-\frac{1}{8},\quad j\in\{9, 10\}\\
w_{4}(\mathbf{t}_{j}) & =\frac{3}{32}, \quad j\in\{15, 17, 24, 26\}\\
w_{4}(\mathbf{t}_{j}) & =\frac{3}{16}, \quad j\in\{16, 25\}
\end{align*} 
which gives $\phi_{4}((3))=1/2$, $\phi_{4}((2,1))+\phi_{4}((1,2))=w_{4}(\mathbf{t}_{3})+w_{4}(\mathbf{t}_{6})=-3/4$, $\phi_{4}((1,1,1))=\sum_{j\in J\setminus\{1,3,6\}}w_{4}(\mathbf{t}_{j})=1/2$, hence
\[
\alpha_{3}=\sum_{\ov{n}\in\mathrm{C}(3)}\phi_{4}(\ov{n})\,\omega(\ov{n})=\frac{1}{2}\,\omega_{3}-\frac{3}{4}\,\omega_{2}\,\omega_{1}+\frac{1}{2}\,\omega_{1}^{3}.
\]

\begin{figure}
\begin{center}
\begin{tikzpicture}[level distance=10mm,
every node/.style={inner sep=1pt,circle,draw},
level 1/.style={sibling distance=10mm},
level 2/.style={sibling distance=6mm}]
\node at (0.2,0) {$3$};
\node[draw=none] at (0.2,-0.5) {$\mathbf{t}_{1}$};
\node at (2.2,0) {$3$}
child {node{$2$} edge from parent node[left, draw=none, scale=0.8] {b}}
child {node{$1$} edge from parent node[right, draw=none, scale=0.8] {b}};
\node[draw=none] at (2.2,-1.5) {$\mathbf{t}_{2}$};
\node at (4.2,0) {$3$}
child {node{$2$} edge from parent node[left, draw=none, scale=0.8] {b}}
child {node{$1$} edge from parent node[right, draw=none, scale=0.8] {r}};
\node[draw=none] at (4.2,-1.5) {$\mathbf{t}_{3}$};
\node at (6.2,0) {$3$}
child {node{$2$} edge from parent node[left, draw=none, scale=0.8] {r}}
child {node{$1$} edge from parent node[right, draw=none, scale=0.8] {r}};
\node[draw=none] at (6.2,-1.5) {$\mathbf{t}_{4}$};
\node at (8.2,0) {$3$}
child {node{$1$} edge from parent node[left, draw=none, scale=0.8] {b}}
child {node{$2$} edge from parent node[right, draw=none, scale=0.8] {b}};
\node[draw=none] at (8.2,-1.5) {$\mathbf{t}_{5}$};
\node at (10.2,0) {$3$}
child {node{$1$} edge from parent node[left, draw=none, scale=0.8] {b}}
child {node{$2$} edge from parent node[right, draw=none, scale=0.8] {r}};
\node[draw=none] at (10.2,-1.5) {$\mathbf{t}_{6}$};
\node at (12.2,0) {$3$}
child {node{$1$} edge from parent node[left, draw=none, scale=0.8] {r}}
child {node{$2$} edge from parent node[right, draw=none, scale=0.8] {r}};
\node[draw=none] at (12.2,-1.5) {$\mathbf{t}_{7}$};
\node at (2,-2.2) {$3$}
child {node{$1$} edge from parent node[left, draw=none, scale=0.8] {b}}
child {node{$1$}}
child {node{$1$} edge from parent node[right, draw=none, scale=0.8] {b}};
\node[scale=0.8, draw=none] at (2.12,-2.8) {b};
\node[draw=none] at (2,-3.7) {$\mathbf{t}_{8}$};
\node at (5,-2.2) {$3$}
child {node{$1$} edge from parent node[left, draw=none, scale=0.8] {b}}
child {node{$1$}}
child {node{$1$} edge from parent node[right, draw=none, scale=0.8] {r}};
\node[scale=0.8, draw=none] at (5.12,-2.8) {b};
\node[draw=none] at (5,-3.7) {$\mathbf{t}_{9}$};
\node at (8,-2.2) {$3$}
child {node{$1$} edge from parent node[left, draw=none, scale=0.8] {b}}
child {node{$1$}}
child {node{$1$} edge from parent node[right, draw=none, scale=0.8] {r}};
\node[scale=0.8, draw=none] at (8.12,-2.8) {r};
\node[draw=none] at (8,-3.7) {$\mathbf{t}_{10}$};
\node at (11,-2.2) {$3$}
child {node{$1$} edge from parent node[left, draw=none, scale=0.8] {r}}
child {node{$1$}}
child {node{$1$} edge from parent node[right, draw=none, scale=0.8] {r}};
\node[scale=0.8, draw=none] at (11.12,-2.8) {r};
\node[draw=none] at (11,-3.7) {$\mathbf{t}_{11}$};
\node at (0.8,-4.4) {$3$}
child {node{$2$}
child{node{$1$} edge from parent node[left, draw=none, scale=0.8]{b}}
child{node{$1$} edge from parent node[right, draw=none, scale=0.8]{b}}
edge from parent node[left, draw=none, scale=0.8]{b}}
child{node{$1$} 
child{node{$1$} edge from parent node[right, draw=none, scale=0.8]{b}}
edge from parent node[right, draw=none, scale=0.8]{b}};
\node[draw=none] at (0.8,-6.9) {$\mathbf{t}_{12}$};
\node at (3.3,-4.4) {$3$}
child {node{$2$}
child{node{$1$} edge from parent node[left, draw=none, scale=0.8]{b}}
child{node{$1$} edge from parent node[right, draw=none, scale=0.8]{r}}
edge from parent node[left, draw=none, scale=0.8]{b}}
child{node{$1$} 
child{node{$1$} edge from parent node[right, draw=none, scale=0.8]{b}}
edge from parent node[right, draw=none, scale=0.8]{b}};
\node[draw=none] at (3.3,-6.9) {$\mathbf{t}_{13}$};
\node at (5.8,-4.4) {$3$}
child {node{$2$}
child{node{$1$} edge from parent node[left, draw=none, scale=0.8]{r}}
child{node{$1$} edge from parent node[right, draw=none, scale=0.8]{r}}
edge from parent node[left, draw=none, scale=0.8]{b}}
child{node{$1$} 
child{node{$1$} edge from parent node[right, draw=none, scale=0.8]{b}}
edge from parent node[right, draw=none, scale=0.8]{b}};
\node[draw=none] at (5.8,-6.9) {$\mathbf{t}_{14}$};
\node at (8.3,-4.4) {$3$}
child {node{$2$}
child{node{$1$} edge from parent node[left, draw=none, scale=0.8]{b}}
child{node{$1$} edge from parent node[right, draw=none, scale=0.8]{b}}
edge from parent node[left, draw=none, scale=0.8]{b}}
child{node{$1$} 
child{node{$1$} edge from parent node[right, draw=none, scale=0.8]{r}}
edge from parent node[right, draw=none, scale=0.8]{r}};
\node[draw=none] at (8.3,-6.9) {$\mathbf{t}_{15}$};
\node at (10.8,-4.4) {$3$}
child {node{$2$}
child{node{$1$} edge from parent node[left, draw=none, scale=0.8]{b}}
child{node{$1$} edge from parent node[right, draw=none, scale=0.8]{r}}
edge from parent node[left, draw=none, scale=0.8]{b}}
child{node{$1$} 
child{node{$1$} edge from parent node[right, draw=none, scale=0.8]{r}}
edge from parent node[right, draw=none, scale=0.8]{r}};
\node[draw=none] at (10.8,-6.9) {$\mathbf{t}_{16}$};
\node at (13.3,-4.4) {$3$}
child {node{$2$}
child{node{$1$} edge from parent node[left, draw=none, scale=0.8]{r}}
child{node{$1$} edge from parent node[right, draw=none, scale=0.8]{r}}
edge from parent node[left, draw=none, scale=0.8]{b}}
child{node{$1$} 
child{node{$1$} edge from parent node[right, draw=none, scale=0.8]{r}}
edge from parent node[right, draw=none, scale=0.8]{r}};
\node[draw=none] at (13.3,-6.9) {$\mathbf{t}_{17}$};
\node at (0.8,-7.6) {$3$}
child {node{$2$}
child{node{$1$} edge from parent node[left, draw=none, scale=0.8]{b}}
child{node{$1$} edge from parent node[right, draw=none, scale=0.8]{b}}
edge from parent node[left, draw=none, scale=0.8]{r}}
child{node{$1$} 
child{node{$1$} edge from parent node[right, draw=none, scale=0.8]{r}}
edge from parent node[right, draw=none, scale=0.8]{r}};
\node[draw=none] at (0.8,-10.1) {$\mathbf{t}_{18}$};
\node at (3.3,-7.6) {$3$}
child {node{$2$}
child{node{$1$} edge from parent node[left, draw=none, scale=0.8]{b}}
child{node{$1$} edge from parent node[right, draw=none, scale=0.8]{r}}
edge from parent node[left, draw=none, scale=0.8]{r}}
child{node{$1$} 
child{node{$1$} edge from parent node[right, draw=none, scale=0.8]{r}}
edge from parent node[right, draw=none, scale=0.8]{r}};
\node[draw=none] at (3.3,-10.1) {$\mathbf{t}_{19}$};
\node at (5.8,-7.6) {$3$}
child {node{$2$}
child{node{$1$} edge from parent node[left, draw=none, scale=0.8]{r}}
child{node{$1$} edge from parent node[right, draw=none, scale=0.8]{r}}
edge from parent node[left, draw=none, scale=0.8]{r}}
child{node{$1$} 
child{node{$1$} edge from parent node[right, draw=none, scale=0.8]{r}}
edge from parent node[right, draw=none, scale=0.8]{r}};
\node[draw=none] at (5.8,-10.1) {$\mathbf{t}_{20}$};
\node at (8.3,-7.6) {$3$}
child{node{$1$} 
child{node{$1$} edge from parent node[left, draw=none, scale=0.8]{b}}
edge from parent node[left, draw=none, scale=0.8]{b}}
child {node{$2$}
child{node{$1$} edge from parent node[left, draw=none, scale=0.8]{b}}
child{node{$1$} edge from parent node[right, draw=none, scale=0.8]{b}}
edge from parent node[right, draw=none, scale=0.8]{b}};
\node[draw=none] at (8.3,-10.1) {$\mathbf{t}_{21}$};
\node at (10.8,-7.6) {$3$}
child{node{$1$} 
child{node{$1$} edge from parent node[left, draw=none, scale=0.8]{b}}
edge from parent node[left, draw=none, scale=0.8]{b}}
child {node{$2$}
child{node{$1$} edge from parent node[left, draw=none, scale=0.8]{b}}
child{node{$1$} edge from parent node[right, draw=none, scale=0.8]{r}}
edge from parent node[right, draw=none, scale=0.8]{b}};
\node[draw=none] at (10.8,-10.1) {$\mathbf{t}_{22}$};
\node at (13.3,-7.6) {$3$}
child{node{$1$} 
child{node{$1$} edge from parent node[left, draw=none, scale=0.8]{b}}
edge from parent node[left, draw=none, scale=0.8]{b}}
child {node{$2$}
child{node{$1$} edge from parent node[left, draw=none, scale=0.8]{r}}
child{node{$1$} edge from parent node[right, draw=none, scale=0.8]{r}}
edge from parent node[right, draw=none, scale=0.8]{b}};
\node[draw=none] at (13.3,-10.1) {$\mathbf{t}_{23}$};
\node at (0.8,-10.8) {$3$}
child{node{$1$} 
child{node{$1$} edge from parent node[left, draw=none, scale=0.8]{b}}
edge from parent node[left, draw=none, scale=0.8]{b}}
child {node{$2$}
child{node{$1$} edge from parent node[left, draw=none, scale=0.8]{b}}
child{node{$1$} edge from parent node[right, draw=none, scale=0.8]{b}}
edge from parent node[right, draw=none, scale=0.8]{r}};
\node[draw=none] at (0.8,-13.3) {$\mathbf{t}_{24}$};
\node at (3.3,-10.8) {$3$}
child{node{$1$} 
child{node{$1$} edge from parent node[left, draw=none, scale=0.8]{b}}
edge from parent node[left, draw=none, scale=0.8]{b}}
child {node{$2$}
child{node{$1$} edge from parent node[left, draw=none, scale=0.8]{b}}
child{node{$1$} edge from parent node[right, draw=none, scale=0.8]{r}}
edge from parent node[right, draw=none, scale=0.8]{r}};
\node[draw=none] at (3.3,-13.3) {$\mathbf{t}_{25}$};
\node at (5.8,-10.8) {$3$}
child{node{$1$} 
child{node{$1$} edge from parent node[left, draw=none, scale=0.8]{b}}
edge from parent node[left, draw=none, scale=0.8]{b}}
child {node{$2$}
child{node{$1$} edge from parent node[left, draw=none, scale=0.8]{r}}
child{node{$1$} edge from parent node[right, draw=none, scale=0.8]{r}}
edge from parent node[right, draw=none, scale=0.8]{r}};
\node[draw=none] at (5.8,-13.3) {$\mathbf{t}_{26}$};
\node at (8.3,-10.8) {$3$}
child{node{$1$} 
child{node{$1$} edge from parent node[left, draw=none, scale=0.8]{r}}
edge from parent node[left, draw=none, scale=0.8]{r}}
child {node{$2$}
child{node{$1$} edge from parent node[left, draw=none, scale=0.8]{b}}
child{node{$1$} edge from parent node[right, draw=none, scale=0.8]{b}}
edge from parent node[right, draw=none, scale=0.8]{r}};
\node[draw=none] at (8.3,-13.3) {$\mathbf{t}_{27}$};
\node at (10.8,-10.8) {$3$}
child{node{$1$} 
child{node{$1$} edge from parent node[left, draw=none, scale=0.8]{r}}
edge from parent node[left, draw=none, scale=0.8]{r}}
child {node{$2$}
child{node{$1$} edge from parent node[left, draw=none, scale=0.8]{b}}
child{node{$1$} edge from parent node[right, draw=none, scale=0.8]{r}}
edge from parent node[right, draw=none, scale=0.8]{r}};
\node[draw=none] at (10.8,-13.3) {$\mathbf{t}_{28}$};
\node at (13.3,-10.8) {$3$}
child{node{$1$} 
child{node{$1$} edge from parent node[left, draw=none, scale=0.8]{r}}
edge from parent node[left, draw=none, scale=0.8]{r}}
child {node{$2$}
child{node{$1$} edge from parent node[left, draw=none, scale=0.8]{r}}
child{node{$1$} edge from parent node[right, draw=none, scale=0.8]{r}}
edge from parent node[right, draw=none, scale=0.8]{r}};
\node[draw=none] at (13.3,-13.3) {$\mathbf{t}_{29}$};
\end{tikzpicture}
\end{center}
\caption{The twenty-nine trees that form the set $\bigcup_{\ov{n}\in\mathrm{C}(3)}\mathcal{T}_{3}(\ov{n}).$}
\label{exampletrees2}
\end{figure}
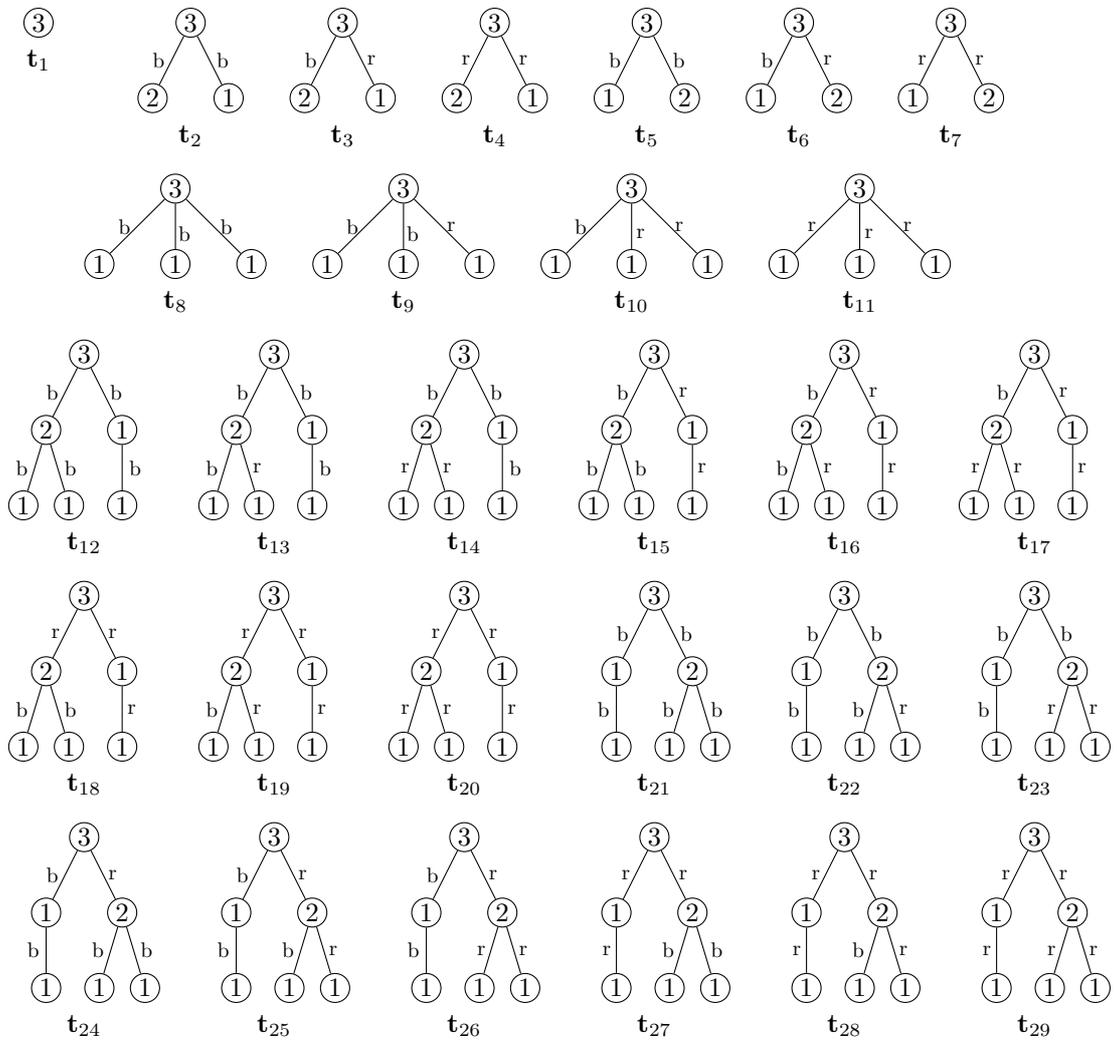

\section{Asymptotics}\label{sec:asymp}

We remind the reader that throughout this work we keep the hypotheses stated at the beginning of the Introduction. Recall also that in subsection~\ref{subs:lattice} we defined $\mathcal{P}(n,k,i)$ as the collection of all paths $\gamma$ on $\mathcal{G}$ satisfying $1\leq \min(\gamma)\leq \max(\gamma)\leq n$ and having initial point $(0,i)$ and ending point $(k,i)$.

\begin{lemma}\label{lem:equivpaths}
Let $m\geq 1$ and $n\geq 1+2m$ be fixed. For each $i$ satisfying $1+m\leq i\leq n-m$, the map $\gamma\mapsto \gamma-i$ is a bijection from $\mathcal{P}(n,2m,i)$ onto $\mathcal{P}_{m}$, where $\mathcal{P}_{m}$ is the collection of all generalized Dyck paths of length $2m$. Similarly, the map $\gamma\mapsto \gamma -1$ is a bijection from $\mathcal{P}(n,2m,1)$ onto $\mathcal{D}_{m}$, where $\mathcal{D}_{m}$ is the collection of all Dyck paths of length $2m$. We have
\begin{align}
\omega_{m} & =\sum_{\gamma\in\mathcal{P}(n,2m,i)}\mathbb{E}(w(\gamma)),\qquad 1+m\leq i\leq n-m. \label{eq:weightPnki}\\
\alpha_{m} & =\sum_{\gamma\in\mathcal{P}(n,2m,1)}\mathbb{E}(w(\gamma)). \label{eq:weightPnk1}
\end{align}
\end{lemma}
\begin{proof}
Assume that $m\geq 1$, $n\geq 1+2m$ and $1+m\leq i\leq n-m$. Recall that any path $\gamma\in\mathcal{P}(n,2m,i)$ has $m$ up steps and $m$ down steps. The restrictions on the initial height $i$ allow to have the $m$ up steps in any position (equivalently the $m$ down steps in any position). For example, the first $m$ edges could be down steps and all subsequent edges up steps (this is not possible if $i<1+m$). It then follows that the map $\gamma\mapsto \gamma-i$ is a bijection from $\mathcal{P}(n,2m,i)$ onto $\mathcal{P}_{m}$. Moreover, since the random variables $(a_{n})_{n\in\mathbb{Z}}$ are independent and identically distributed, we have
\[
\mathbb{E}(w(\gamma))=\mathbb{E}(w(\gamma-i)),\qquad\mbox{for all}\,\,\gamma\in\mathcal{P}(n,2m,i).
\]
Hence, in view of \eqref{def:Wm} and \eqref{def:omegam},
\[
\sum_{\gamma\in\mathcal{P}(n,2m,i)}\mathbb{E}(w(\gamma))=\sum_{\gamma\in\mathcal{P}(n,2m,i)}\mathbb{E}(w(\gamma-i))=\sum_{\gamma\in\mathcal{P}_{m}}\mathbb{E}(w(\gamma))=\omega_{m},
\]
which establishes \eqref{eq:weightPnki}.

If $n\geq m+1$, then the map $\gamma\mapsto \gamma -1$ is clearly a bijection from $\mathcal{P}(n,2m,1)$ onto $\mathcal{D}_{m}$, and \eqref{eq:weightPnk1} follows from \eqref{def:Am} and \eqref{def:alpham}.\end{proof}

\begin{theorem}
Let $k\in\mathbb{Z}_{\geq 0}$ be fixed, and let $H_{n}$ be the tridiagonal matrix defined in \eqref{def:Hn}. Then,
\begin{align}
\lim_{n\rightarrow\infty}\frac{1}{n}\mathbb{E}(\Tr(H_{n}^{k})) & =\begin{cases}
0, & \mbox{if}\,\,k\,\,\mbox{is odd},\\
\omega_{k/2}, & \mbox{if}\,\,k\,\,\mbox{is even},
\end{cases}\label{eq:asymptrace}\\
\lim_{n\rightarrow\infty}\mathbb{E}(H_{n}^{k}(1,1)) & =\begin{cases}
0, & \mbox{if}\,\,k\,\,\mbox{is odd},\\
\alpha_{k/2}, & \mbox{if}\,\,k\,\,\mbox{is even}.
\end{cases}\label{eq:asymp11entry}
\end{align}
\end{theorem}
\begin{proof}
If $k$ is odd then the result follows from \eqref{nulltrace11entry}. Suppose now that $k=2m$, $m\geq 1$. According to \eqref{eq:trace:rep}, we have
\[
\frac{1}{n}\mathbb{E}(\Tr(H_{n}^{2m}))=\frac{1}{n}\sum_{i=1}^{n}\sum_{\gamma\in\mathcal{P}(n,2m,i)}\mathbb{E}(w(\gamma)).
\]
If $n\geq 1+2m$, then
\[
\sum_{\gamma\in\mathcal{P}(n,2m,i)}\mathbb{E}(w(\gamma))=\omega_{m},\qquad 1+m\leq i\leq n-m,
\]
and for each $1\leq i\leq n$, the expression $\sum_{\gamma\in\mathcal{P}(n,2m,i)}\mathbb{E}(w(\gamma))$ is a constant that only depends on $m$ (not on $n$). Therefore
\[
\frac{1}{n}\mathbb{E}(\Tr(H_{n}^{k}))=\frac{n-2m}{n} \omega_{m}+o(1),\quad n\rightarrow \infty,
\]
and \eqref{eq:asymptrace} follows.

In virtue of \eqref{eq:entry11} and \eqref{eq:weightPnk1}, for $n\geq 2m+1$,
\[
\mathbb{E}(H_{n}^{2m}(1,1))=\sum_{\gamma\in\mathcal{P}(n,2m,1)}\mathbb{E}(w(\gamma))=\alpha_{m},
\]
hence \eqref{eq:asymp11entry} follows.\end{proof}

Let $\sigma_{n}$ and $\tau_{n}$ be the discrete random measures indicated in \eqref{def:randomsigman} and \eqref{def:randomtaun}. We consider the averages $\mathbb{E} \sigma_{n}$ and $\mathbb{E} \tau_{n}$, which are the measures satisfying the duality identities 
\begin{align*}
\int f d\mathbb{E} \sigma_{n} & =\mathbb{E} \int f d\sigma_{n},\\
\int f d\mathbb{E} \tau_{n} & =\mathbb{E} \int f d\tau_{n},
\end{align*}
for all $f$ in the space $C_{b}(\mathbb{R})$ of all real-valued bounded continuous functions on $\mathbb{R}$.

\begin{corollary}
Assume there exist probability measures $\sigma$ and $\tau$ on $\mathbb{R}$ with moments of all orders finite, such that for all $k\in\mathbb{Z}_{\geq 0}$,
\begin{align*}
\int x^{k}\,d\sigma(x) & = \begin{cases}
0, & \mbox{if}\,\,k\,\,\mbox{is odd},\\
\omega_{k/2}, & \mbox{if}\,\,k\,\,\mbox{is even},
\end{cases} \\
\int x^{k}\,d\tau(x) & =\begin{cases}
0, & \mbox{if}\,\,k\,\,\mbox{is odd},\\
\alpha_{k/2}, & \mbox{if}\,\,k\,\,\mbox{is even}.
\end{cases}
\end{align*}
Assume further that the measures $\sigma$ and $\tau$ are uniquely determined by their moments. Then, the sequences $(\mathbb{E} \sigma_{n})_{n=1}^{\infty}$ and $(\mathbb{E}\tau_{n})_{n=1}^{\infty}$ converge weakly to $\sigma$ and $\tau$, respectively. That is, for any $f\in C_{b}(\mathbb{R})$, we have
\begin{align}
\lim_{n\rightarrow\infty} \int f d\mathbb{E}\sigma_{n} & =\int f d\sigma,\label{weaklimsigma}\\
\lim_{n\rightarrow\infty} \int f d\mathbb{E}\tau_{n} & =\int f d\tau.\label{weaklimtau}
\end{align}
Moreover, these limits also hold for a continuous function $f$ for which there exists a polynomial $P$ such that $|f(x)|\leq P(x)$ for all $x\in\mathbb{R}$.
\end{corollary}
\begin{proof}
Using elementary measure theoretic arguments, the reader can easily check that for each $n$, the measures $\mathbb{E} \sigma_{n}$ and $\mathbb{E} \tau_{n}$ have moments of all orders finite, and
\begin{align*}
\int x^{k} d \mathbb{E}\sigma_{n}(x) & =\mathbb{E} \int x^{k} d\sigma_{n}(x),\qquad k\in\mathbb{Z}_{\geq 0},\\
\int x^{k} d \mathbb{E}\tau_{n}(x) & =\mathbb{E} \int x^{k} d\tau_{n}(x),\qquad k\in\mathbb{Z}_{\geq 0}.
\end{align*}
Hence, combining \eqref{relsigmatrace}--\eqref{reltauH} with \eqref{eq:asymptrace}--\eqref{eq:asymp11entry}, we obtain
\begin{align*}
\lim_{n\rightarrow\infty} \int x^{k} d\mathbb{E}\sigma_{n}(x) & =\int x^{k} d\sigma(x), \qquad k\in\mathbb{Z}_{\geq 0},\\
\lim_{n\rightarrow\infty} \int x^{k} d\mathbb{E}\tau_{n}(x) & =\int x^{k} d\tau(x), \qquad k\in\mathbb{Z}_{\geq 0}.
\end{align*}
Thus, \eqref{weaklimsigma}--\eqref{weaklimtau} follows after applying Lemma~\ref{aux:lemmaprob}. See \cite[Lemma 2.1]{Khan} for a justification of the claim in the case of functions $f$ with polynomial growth.
\end{proof}

\section{Appendix}

Recall that for a formal Laurent series $S(z)\in\mathbb{C}((z^{-1}))$, the symbol $[S]_{n}$ refers to the coefficient of $z^{-n}$ in the expression of $S(z)$.

\begin{lemma}\label{lemma:app:1}
Let $S(z)$ be the series
\begin{equation}\label{form:S}
S(z)=\sum_{n=0}^{\infty}\frac{s_{n}}{z^{2n+1}}.
\end{equation}
Then the series $R(z)$ given by
\begin{equation}\label{eq:relRS}
R(z)=\frac{1}{z-S(z)}
\end{equation}
is also of the form
\begin{equation}\label{eq:formR}
R(z)=\sum_{n=0}^{\infty}\frac{r_{n}}{z^{2n+1}},
\end{equation}
and we have
\begin{equation}\label{rel:coeff:RS}
r_{n}=\sum_{k=0}^{n}[S^{k}]_{2n-k},\qquad n\geq 0.
\end{equation}
More generally, for any $k\geq 1$, $m\geq 0$,
\begin{equation}\label{coeffpowerR}
[R^{k}]_{2m+k}=\sum_{n=0}^{m}\binom{n+k-1}{k-1}[S^{n}]_{2m-n}.
\end{equation}
\end{lemma}
\begin{proof}
From \eqref{eq:relRS}, or equivalently
\[
zR(z)-1=R(z)S(z),
\]
we easily deduce that $R(z)$ is of the form \eqref{eq:formR} and
\begin{equation}
\begin{aligned}
r_{0} & = 1,\\
r_{n} & =\sum_{k=0}^{n-1} s_{k}\ r_{n-k-1} ,\qquad n\geq 1.
\end{aligned}
\label{eq:rel:rs}
\end{equation}
Now we prove \eqref{rel:coeff:RS} by induction.

First, \eqref{rel:coeff:RS} for $n=0$ is simply that $r_{0}=1$. Assume that \eqref{rel:coeff:RS} holds for every $n=0,\ldots,\ell$, and let us prove that this identity is valid for $n=\ell+1$. We have
\begin{equation}\label{eq:aux:1}
r_{\ell+1} =\sum_{k=0}^{\ell} s_{k} \ r_{\ell-k}=\sum_{k=0}^{\ell}\sum_{i=0}^{\ell-k} s_{k} [S^{i}]_{2\ell-2k-i}=\sum_{i=0}^{\ell}\sum_{k=0}^{\ell-i} s_{k} [S^{i}]_{2\ell-2k-i}
\end{equation}
where in the second equality we used the induction hypothesis. On the other hand, from \eqref{form:S} we obtain immediately that the powers of $S(z)$ are of the form
\[
S^{i}(z)=\sum_{n=0}^{\infty}\frac{[S^{i}]_{2n+i}}{z^{2n+i}},\qquad i\geq 0,
\]
and therefore
\[
S^{i+1}(z)=S^{i}(z) S(z)=\sum_{n=0}^{\infty}\frac{[S^{i}]_{2n+i}}{z^{2n+i}}\ \sum_{k=0}^{\infty} \frac{s_{k}}{z^{2k+1}},
\]
which gives the relations
\begin{equation}\label{eq:aux:2}
[S^{i+1}]_{2n+i+1}=\sum_{k=0}^{n}s_{k} [S^{i}]_{2n+i-2k},\qquad i, n\geq 0.
\end{equation}
Applying \eqref{eq:aux:1} and \eqref{eq:aux:2} for $n=\ell-i$, we get
\[
r_{\ell+1}=\sum_{i=0}^{\ell}\sum_{k=0}^{\ell-i} s_{k} [S^{i}]_{2\ell-2k-i}
=\sum_{i=0}^{\ell}[S^{i+1}]_{2\ell+1-i}=\sum_{k=0}^{\ell+1}[S^{k}]_{2(\ell+1)-k}
\]
where in the last equality we used the fact that $[S^{0}]_{2\ell+2}=0$, and this concludes the proof of \eqref{rel:coeff:RS}.

Finally, we prove \eqref{coeffpowerR}, again by induction. This identity for $k=1$ is precisely \eqref{rel:coeff:RS}. Assume that \eqref{coeffpowerR} holds and let us prove that
\begin{equation}\label{tobeproved}
[R^{k+1}]_{2m+k+1}=\sum_{n=0}^{m}\binom{n+k}{k}\,[S^{n}]_{2m-n}.
\end{equation}
Making use of \eqref{eq:aux:2} (applied to $R$), \eqref{rel:coeff:RS} and \eqref{coeffpowerR}, we obtain
\begin{align*}
[R^{k+1}]_{2m+k+1} & =\sum_{j=0}^{m} r_{j}\,[R^{k}]_{2(m-j)+k}\\
& =\sum_{j=0}^{m}\sum_{\ell=0}^{j}\sum_{n=0}^{m-j}\binom{n+k-1}{k-1}\,[S^{\ell}]_{2j-\ell}\,[S^{n}]_{2(m-j)-n}\\
& =\sum_{\ell=0}^{m}\sum_{n=0}^{m-\ell}\sum_{j=\ell}^{m-n}\binom{n+k-1}{k-1}\,[S^{\ell}]_{2j-\ell}\,[S^{n}]_{2(m-j)-n},
\end{align*}
where in the last equality we transposed the summations twice. This and the relation
\[
[S^{\ell+n}]_{2m-(\ell+n)}=\sum_{j=\ell}^{m-n}[S^{\ell}]_{2j-\ell}\,[S^{n}]_{2(m-j)-n}
\]
give
\[
[R^{k+1}]_{2m+k+1}=\sum_{\ell=0}^{m}\sum_{n=0}^{m-\ell}\binom{n+k-1}{k-1}[S^{\ell+n}]_{2m-(\ell+n)}.
\]
Considering the different ways we can add $\ell$ and $n$ to obtain $t=\ell+n$, we deduce from the previous relation that
\[
[R^{k+1}]_{2m+k+1}=\sum_{t=0}^{m}\left(\sum_{n=0}^{t}\binom{n+k-1}{k-1}\right)[S^{t}]_{2m-t}
=\sum_{t=0}^{m}\binom{t+k}{k}[S^{t}]_{2m-t},
\]
which proves \eqref{tobeproved}.
\end{proof}

The following is a well-known result in the theory of moment problems. For a proof, see \cite[Theorem 30.2]{Bill}.

\begin{lemma}\label{aux:lemmaprob}
Let $(\mu_{n})_{n=1}^{\infty}$ and $\mu$ be probability measures on $\mathbb{R}$ with moments of all orders finite, and suppose that for each $k\in\mathbb{Z}_{\geq 0}$,
\[
\lim_{n\rightarrow\infty}\int x^{k} d\mu_{n}(x)=\int x^{k} d\mu(x).
\]
Assume further that the measure $\mu$ is uniquely determined by its moments. Then $\mu_{n}$ converges weakly to $\mu$, i.e., for any $f\in C_{b}(\mathbb{R})$ we have
\[
\lim_{n\rightarrow\infty} \int f d\mu_{n}=\int f d\mu.
\]
\end{lemma} 

\bigskip

\noindent\textbf{Acknowledgments:} We thank Jeff Geronimo and Mourad Ismail for bringing to our attention some pertinent works on random polynomials and Jacobi matrices. The first author acknowledges partial support from the grant MTM2015-65888-C4-2-P of the Spanish Ministry of Economy and Competitiveness.

\end{document}